\documentclass[11pt,reqno]{amsart}
\usepackage{amssymb}
\usepackage{enumerate}
\usepackage{amsmath,amssymb,amsfonts,amsthm,
latexsym, amscd, epsfig, epic,color}
\usepackage{epstopdf}
\usepackage{mathtools}
\usepackage{youngtab}
\usepackage{extarrows}
\usepackage{enumitem}
\usepackage{enumerate}
\usepackage{mathrsfs}
\usepackage{eucal}
\usepackage{eufrak}
\usepackage{xspace}
\usepackage{a4wide}
\usepackage{colordvi}
\usepackage{comment}
\usepackage{hyperref,eucal}
\usepackage{tabularx}
\usepackage{tikz}\usetikzlibrary{matrix,arrows}
\usepackage{ytableau} \usepackage{chemfig}
\usepackage{float}
\usepackage{youngtab}
\usepackage{mathrsfs}
\usepackage{caption}
\usepackage{epsfig}
\usepackage{stfloats}
\usepackage{booktabs}
\usepackage{changepage}
\usepackage{graphicx}
\theoremstyle{plain}

\theoremstyle{plain}

\newtheorem{theorem}{Theorem}[section]

\newtheorem{lemma}{Lemma}[section]

\newtheorem{definition}{Definition}[section]

\theoremstyle{example}
\newtheorem{example}{Example}[section]

\theoremstyle{openproblem}

\def\T{\CMcal{T}}

\def\R{\CMcal{R}}

\def\A{\mathcal{A}}
\def\B{\mathcal{B}}
\def\S{\CMcal{S}}
\def\inv{\mathsf{inv}}
\def\quinv{\mathsf{quinv}}
\def\rev{\mathsf{rev}}
\def\maj{\mathsf{maj}}
\def\ndes{\mathsf{ndes}}
\def\Des{\mathsf{Des}}
\def\des{\mathsf{des}}
\def\South{\mathsf{South}}
\def\ndes{\mathsf{ndes}}

\def\maxi{\mathsf{max}}
\def\mini{\mathsf{min}}

\makeatletter
\@namedef{subjclassname@2020}{\textup{2020} Mathematics Subject Classification}
\makeatother

\theoremstyle{remark}

\newtheorem{remark}{Remark}
\numberwithin{equation}{section}
\numberwithin{figure}{section}

\title[Modified Macdonald polynomials and $\mu$-Mahonian statistics]{Modified Macdonald polynomials and $\mu$-Mahonian statistics}

\author{Emma Yu Jin}
\address{School of Mathematical Sciences, Xiamen University, Xiamen 361005, China}
\email{yjin@xmu.edu.cn}

\author{Xiaowei Lin}
\address{School of Mathematical Sciences, Xiamen University, Xiamen 361005, China}
\email{linxiaoweiqing@126.com}



\begin{document}

\begin{abstract}
The Haglund--Haiman--Loehr theorem provides the following combinatorial formula for the modified Macdonald polynomials:
\begin{align*}
\tilde{H}_{\mu}(X;q,t)=\sum_{\sigma: \mu\rightarrow \mathbb{P}}x^{\sigma}t^{\maj(\sigma)}q^{\inv(\sigma)}.
\end{align*}
Inspired by Martin's multiline-queue formula for the stationary distribution of multitype asymmetric simple exclusion processes, Corteel, Haglund, Mandelshtam, Mason and Williams recently introduced the queue inversion statistic $\quinv$ and conjectured that the tableaux formula for $\tilde{H}_{\mu}(X;q,t)$ is invariant if the inversion statistic $\inv$ is replaced by $\quinv$. This was subsequently resolved by Ayyer, Mandelshtam and Martin, who proposed a stronger conjecture on the equivalence of the two refined formulas for $\tilde{H}_{\mu}(X;q,t)$.

Our main result confirms this Ayyer--Mandelshtam--Martin conjecture. We establish an equidistribution between the pairs $(\inv,\maj)$ and $(\quinv,\maj)$ of $\mu$-Mahonian statistics on any row-equivalency class $[\tau]$, where $\tau$ is a filling of the Young diagram of $\mu$. As a byproduct of our approach, we show that if $\tau$ is a rectangular filling, the triples $(\inv,\quinv,\maj)$ and $(\quinv,\inv,\maj)$ have the same distribution over $[\tau]$.

\end{abstract}
\maketitle

\section{Introduction}
{\em Macdonald polynomials} $P_{\mu}(X;q,t)$ indexed by partitions are polynomials in infinitely many variables $X=\{x_1,x_2,\ldots\}$ with coefficients in the field $\mathbb{Q}(q,t)$ of rational functions of two variables $q$ and $t$. Several important classes of symmetric polynomials are well-studied specializations of Macdonald polynomials such as
Schur polynomials (when $q=t$), Hall--Littlewood polynomials (when $q=0$) and Jack polynomials (when $q=t^{\alpha}$ and let $t\rightarrow 1$); see for example \cite{Mac88,Mac95}.

Macdonald polynomials $P_{\mu}(X;q,t)$ are defined as the unique basis for the ring of symmetric functions over the field $\mathbb{Q}(q,t)$ satisfying certain orthogonality and lower triangularity properties. The orthogonality is defined with respect to the Hall scalar product, and the triangularity refers to the expansion of $P_{\mu}(X;q,t)$ in terms of the monomial symmetric functions $m_{\lambda}(X)$. Since the coefficients in this expansion may have nontrivial denominators, Macdonald introduced the {\em integral form} of $P_{\mu}(X;q,t)$, denoted by $J_{\mu}(X;q,t)$, which is defined as the product of $P_{\mu}(X;q,t)$ and a polynomial determined by the arm- and leg-lengths of the Young diagram of $\mu$ (see \cite[Equation (6.3)]{Mac88} and \cite[Chapter VI, Equation (8.3)]{Mac95}). Equivalently,
\begin{align}\label{E:j1}
J_{\mu}(X;q,t)=\sum_{\lambda}K_{\lambda\mu}(q,t)s_{\lambda}[X(1-t)],
\end{align}
where $f[X]$ denotes the plethystic substitution of $X$ into the symmetric function $f$, $s_{\lambda}$ denotes the Schur function, and $K_{\lambda\mu}(q,t)$ is the $(q,t)$-Kostka number.

Subsequently, another widely studied variant of Macdonald polynomials, called the {\em modified Macdonald polynomials} $\tilde{H}_{\mu}(X;q,t)$ (also known as the {\em transformed Macdonald polynomials}), was introduced by Garsia and Haiman \cite{GH93}. 

Let $\tilde{K}_{\lambda\mu}(q,t)=t^{n(\mu)}K_{\lambda\mu}(q,t^{-1})$, where $n(\mu)=\sum_{i}(i-1)\mu_i$. Then,
\begin{align*}
\tilde{H}_{\mu}(X;q,t)=\sum_{\lambda}\tilde{K}_{\lambda\mu}(q,t)s_{\lambda}(X).
\end{align*}
 Remarkably, Haiman discovered that $\tilde{H}_{\mu}(X;q,t)$ equals the Frobenius series of a space spanned by certain polynomials together with all of their partial derivatives \cite[Equation (57)]{Haiman01}. In parallel, the combinatorial investigation of modified Macdonald polynomials has been greatly advanced by the celebrated breakthrough on the surprising connections between $\tilde{H}_{\mu}(X;q,t)$ and the combinatorial statistics $\maj$ and $\inv$ on fillings of Young diagrams, due to Haglund, Haiman and Loehr \cite{HHL04}. 

A filling of the Young diagram of $\mu$ is a function $\sigma: \mu\rightarrow \mathbb{P}$ (where $\mathbb{P} $ is the set of positive integers) such that each cell is filled with a positive integer. Let $x^{\sigma}$ denote the content of the filling $\sigma$. Then, Haglund, Haiman and Loehr \cite{HHL04} established that 

\begin{align}\label{E:mmp1}
\tilde{H}_{\mu}(X;q,t)=\sum_{\sigma: \mu\rightarrow \mathbb{P}}x^{\sigma}t^{\maj(\sigma)}q^{\inv(\sigma)},
\end{align}
where $\maj$ and $\inv$ are natural extensions of the major index and inversion number for permutations, respectively. The reading word $\omega(\sigma)$ of $\sigma$ is defined as the sequence formed by the entries in the top row, read from left to right, followed by the entries in the next row from left to right, and so on.
When the Young diagram consists of a single row (resp. a single column), we have
$\inv(\sigma)=\inv(\omega(\sigma))$ (resp. $\maj(\sigma)=\maj(\omega(\sigma))$), indicating that the statistic $\inv$ (resp. $\maj$) reduces to the classical Mahonian statistic on permutations of a multiset \cite{Haglund:book,Stanley}. In general, the statistics $\inv$ and $\maj$ defined on the fillings of the Young diagram of $\mu$ are called $\mu$-Mahonian statistics.
Their precise definitions are presented in Section \ref{S:2}. 

Recently, Corteel, Haglund, Mandelshtam, Mason and Williams \cite[Theorem 3.5]{CHO22} proved a compact formula for $\tilde{H}_{\mu}(X;q,t)$, expressed as a sum over sorted tableaux, and they conjectured the following equivalent form of (\ref{E:mmp1}):
\begin{align}\label{E:mmp2}
\tilde{H}_{\mu}(X;q,t)=\sum_{\sigma:\mu\rightarrow \mathbb{P}}x^{\sigma}t^{\maj(\sigma)}q^{\quinv(\sigma)},
\end{align}
where $\quinv$ is called {\em queue inversion}, as defined in Section \ref{S:2}. This statistic was inspired by Martin's multiline-queue formula for the stationary distribution of multitype asymmetric simple exclusion processes \cite[Theorem 3.4]{Martin:20}. This conjecture was confirmed by Ayyer, Mandelshtam and Martin \cite[Equation (2)]{AMM23} who verified that the right-hand side of (\ref{E:mmp2}) satisfies certain orthogonal and triangular conditions that uniquely characterize the modified Macdonald polynomials $\tilde{H}_{\mu}(X;q,t)$. 

Interestingly, a refinement of the equality between (\ref{E:mmp1}) and (\ref{E:mmp2}) was conjectured by Ayyer, Mandelshtam and Martin (Conjecture 10.3 in \cite{AMM23}).
Our main result affirms this conjecture; see (\ref{E:T1}) of Theorem \ref{T:3}. As a byproduct of our approach, we establish the equidistribution (\ref{E:T2}) between the triples $(\inv,\quinv,\maj)$ and $(\quinv,\inv,\maj)$ for all fillings of a rectangular shape $\mu=(n^m)$. Two fillings $\sigma$ and $\tau$ of a given diagram are called {\em row-equivalent}, denoted by $\sigma\sim\tau$, if the multisets of entries in the $i$th row of $\sigma$ and $\tau$ are identical for all $i$. Our main result is the following:
\begin{theorem}\label{T:3}
	 Let $[\sigma]$ denote the row-equivalency class of $\sigma$. Then,
\begin{align}\label{E:T1}
\sum_{\tau\in [\sigma]}t^{\maj(\tau)}q^{\inv(\tau)}=\sum_{\tau\in [\sigma]}t^{\maj(\tau)}q^{\quinv(\tau)}.
\end{align}	
If $\sigma$ is a rectangular filling, then
\begin{align}\label{E:T2}
\sum_{\tau\in [\sigma]}t^{\maj(\tau)}q^{\inv(\tau)}u^{\quinv(\tau)}=\sum_{\tau\in [\sigma]}t^{\maj(\tau)}u^{\inv(\tau)}q^{\quinv(\tau)}.
\end{align}
\end{theorem}
Note that the symmetric distribution (\ref{E:T2}) does not hold for an arbitrary filling $\sigma$; we provide such an example in Remark \ref{Tab:counterexample}. The proof of Theorem \ref{T:3} makes significant use of two operators: the reverse operator and the flip operator (also called the column-switching operator).

\begin{remark}
Bhattacharya, Ratheesh, and Viswanath (see \cite[Theorem 1]{BRV23} and \cite[Theorems 3.6 and 4.5]{R24}) proved three special cases of (\ref{E:T1}); namely, the case when the entries of $\sigma$ in each column strictly decrease from top to bottom (in the French representation of the Young diagram), the case when $q=0$, and the case when the fillings $\tau$ maximize the number of inversions or queue inversions. 
	
Their proofs are also bijective, developing novel connections between different combinatorial models, maps and statistics such as the Gelfand--Tsetlin patterns, partitions overlaid patterns, box complementation, and the statistics charge and cocharge on words.
\end{remark}

The rest of the paper is organized as follows: in Section \ref{S:2}, we present preliminaries on the $\mu$-Mahonian statistics $\inv$, $\quinv$, and $\maj$ of fillings. In Section \ref{S:2operators}, the reverse and flip operators, which are the key ingredients of our proof, are described. In Section \ref{S:mainproof}, we provide a roadmap for the proof, and Sections \ref{S:5} and \ref{S:6} are devoted to proving Theorem \ref{T:3}.

\section{Preliminaries and notations}\label{S:2}
A partition $\mu=(\mu_1,\cdots,\mu_k)$ of $n$ is a sequence of positive integers such that $\mu_i\ge \mu_{i+1}$ for all $1\le i<k$ and $|\mu|=\mu_1+\cdots+\mu_k=n$. Integer $\mu_i$ is called the $i$th part of $\mu$, and $k$ is called the length of $\mu$, denoted by $\ell(\mu)$. The Young diagram of $\mu$ is an array of boxes with $\mu_i$ boxes in the $i$th row from bottom to top, with the first box of each row left-justified (see Figure \ref{F:1}). We write $\mu=(n^m)$ if $\mu$ is the partition with exactly $m$ parts of size $n$.

A box has coordinates $(i,j)$ if it is in the $i$th row from the bottom and the $j$th column from the left. Let $\mu'$ denote the transpose of $\mu$; that is, the Young diagram of $\mu'$ is obtained from that of $\mu$ by reflecting across the main diagonal (boxes with coordinates $(i,i)$); see Figure \ref{F:1}.
\begin{figure}[H]
\begin{ytableau}[] & \\  & & & \\ & & &   \end{ytableau}
\quad\quad\quad
\begin{ytableau}[] & \\  &   \\ & & \\ & &   \end{ytableau}	
	\caption{\label{F:1} The Young diagrams of $\mu=(4,4,2)$ (left) and its transpose $\mu'=(3,3,2,2)$ (right).}
\end{figure}
A filling of the Young diagram of $\mu$ is a function $\sigma :\mu\rightarrow \mathbb{P}$ (where $\mathbb{P}$ is the set of positive integers), which assigns to each box $z$ of the Young diagram of $\mu$ (denoted by $z\in \mu$) to a positive integer $\sigma(z)$. We use $\South(z)$ to denote the box immediately below $z$. Let $\T(\mu)$ be the set of fillings of the Young diagram of $\mu$, and define
\begin{align*}
x^{\sigma}=\prod_{z\in \mu}x_{\sigma(z)}
\end{align*}
to be the monomial corresponding to the content of $\sigma$. A {\em descent} (resp. {\em non-descent}) of a filling $\sigma \in \T(\mu)$ is a pair of entries $(\sigma(z),\sigma(\South(z))$ such that $\sigma(z)>\sigma(\South(z))$ (resp. $\sigma(z)\le \sigma(\South(z))$).
We define $\Des(\sigma)=\{z\in \mu: \sigma(z)>\sigma(\South(z))\}$ to be the {\em descent set} of $\sigma$, and $\des(\sigma)=|\Des(\sigma)|$.
Let $\mathsf{leg}(z)$ be the number of boxes strictly above box $z$ in its column. Then, 
\begin{align*}
\maj(\sigma)=\sum_{z\in\Des(\sigma)}({\mathsf{leg}}(z)+1)
\end{align*}
is called {\em the major index} of $\sigma$. 
We define $N\des(\sigma)=(a_1,\ldots,a_k)$, where $a_i$ is the number of non-descents in column $i$ of $\sigma$, and set $\ndes(\sigma)=|{N}\des(\sigma)|=a_1+\cdots+a_k=|\{z\in \mu: \sigma(z)\le \sigma(\South(z))\}|$ to be the number of non-descents of $\sigma$. 

Given a filling $\sigma\in \T(\mu)$, let $\sigma(\mu_i'+1,i)=0$ for all $1\le i\le \mu_1$. In other words, we add a box with entry $0$ above the topmost box of each column of $\sigma$. Furthermore, let $\sigma(0,i)=\infty$ for all $1\le i\le \mu_1$. That is, we add a box with entry $\infty$ below the bottommost box of each column of $\sigma$. 

We define $\CMcal{Q}(a,b,c)=1$ if one of the conditions $a<b<c$, $b<c<a$, $c<a<b$ or $a=b\ne c$ holds;
otherwise, $\CMcal{Q}(a,b,c)=0$. A {\em queue inversion triple} of $\sigma$ is a triple $(a,b,c)$ of entries in $\sigma$ such that (as illustrated on the left of the diagram)

\begin{enumerate}
	\item $b$ and $c$ are in the same row, with $c$ to the right of $b$;
	\item $a$ and $b$ are in the same column, with $b$ immediately below $a$;
	\item $\CMcal{Q}(a,b,c)=1$, with $c\ne 0$.
\end{enumerate}

An {\em inversion triple} of $\sigma$ is a triple $(a,b,c)$ of entries in $\sigma$ satisfying (2), (3) and
\begin{enumerate}
	\setcounter{enumi}{3}
    \item $a$ and $c$ are in the same row, with $c$ to the right of $a$.
\end{enumerate}
This is illustrated on the right of the diagram.

\begin{figure}[H]
 \centering
	\begin{minipage}[H]{0.45\linewidth}
 \centering
		\begin{ytableau}
		a  &\none[]
		\\
		b &\none[\dots] & c\\
	\end{ytableau}
\caption*{queue inversion triple}
	\end{minipage}
	\begin{minipage}[H]{0.4\linewidth}
 \centering
		\begin{ytableau}
			a  &\none[\dots] & c
			\\
			b
			&  \none[]\\
		\end{ytableau}
\caption*{inversion triple}
	\end{minipage}
\end{figure}

Let quinv$(\sigma)$ and inv$(\sigma)$ be the numbers of queue inversion triples and inversion triples of $\sigma$, respectively. Following the same strategy to prove Equation (6.1) of \cite{LN12}, Theorem \ref{T:3} can be proved for $t=1$; that is, for $\sigma\in \T(\mu)$,
\begin{align}\label{E:invq}
\sum_{\tau\in [\sigma]}q^{\inv(\tau)}=\sum_{\tau\in [\sigma]}q^{\quinv(\tau)}=\prod_{i=1}^{\ell(\mu)}{a_{i,1}+\cdots+a_{i,N} \brack a_{i,1},\ldots,a_{i,N}}_q,
\end{align}	
where the $i$th row of $\sigma$ contains exactly $a_{i,1}$ copies of $1$, $a_{i,2}$ copies of $2$, and so on. The rightmost term of (\ref{E:invq}) is a product of $q$-multinomial coefficients (see for instance \cite{Loehr,Stanley}). 
The original motivation in \cite{LN12} was to bijectively prove an equality involving a specialization of modified Macdonald polynomials via the inversion flip operator. It turns out that such an operator is of central importance in deducing a compact combinatorial formula for the modified Macdonald polynomials \cite[Theorem 3.5]{CHO22}.

Let us now recall a formula (Equation (6.2) from \cite{LN12}) to explain the concept of $\mu$-Mahonian statistics. Let $\tau'$ be the filling obtained by transposing the filling $\tau$. In other words, $\tau\in \T(\mu)$ if and only if $\tau'\in \T(\mu')$.
Then, the generating function of $\maj(\tau')$ over $\tau\in[\sigma]$ also equals (\ref{E:invq}); namely,
\begin{align*}
\sum_{\tau\in [\sigma]}q^{\inv(\tau)}=\sum_{\tau\in [\sigma]}q^{\maj(\tau')}.
\end{align*}	
Throughout this paper, any statistic whose distribution over $\T(\mu)$ or $\T(\mu')$ equals the distribution of $\inv$ on $\T(\mu)$ is called a {\em $\mu$-Mahonian statistic}. Therefore, the statistics $\quinv$, $\inv$, and $\maj$ are $\mu$-Mahonian statistics.


\section{Reverse and flip operators}\label{S:2operators}
This section is focused on two operators: the reverse operator and the queue inversion flip operator introduced by Ayyer, Mandelshtam and Martin \cite[Definition 7.3]{AMM23}. The latter was inspired by the inversion flip operator proposed by Loehr and Niese \cite[Definition 5.1]{LN12}.

Both operators are related to the decomposition of the Young diagram of $\mu$ into rectangles \cite{CHO22}. Each Young diagram of $\mu$ is regarded as a concatenation of
maximal rectangles, with heights strictly decreasing from left to right. For any $\sigma\in\T(\mu)$, let $\sigma_i$ denote the filling of the $i$th rectangle of the Young diagram of $\mu$ from left to right. We write $\sigma=\sigma_1 \sqcup \cdots \sqcup \sigma_p$, where $p$ is the number of distinct parts of $\mu'$, and $\mu_1'$ is the height of $\sigma_1$. See Figure \ref{F:b2} for an example.
\begin{figure}[ht]
	\centering
	\includegraphics[scale=0.7]{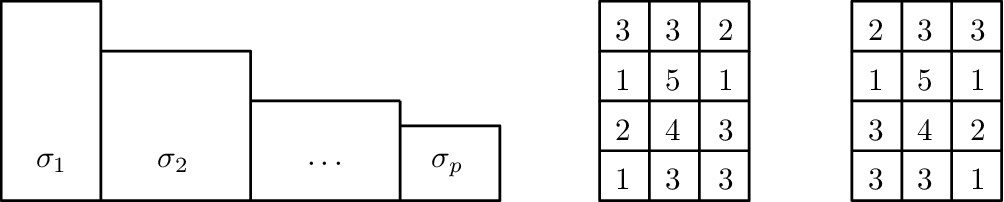}
	\caption{\label{F:b2} A decomposition of $\sigma=\sigma_1 \sqcup\sigma_2\sqcup \cdots \sqcup \sigma_p$ (left), a rectangular filling $\sigma$ (middle) and its reverse (right).}
\end{figure}

\begin{definition}[Reverse operator $\rev$]
For a partition $\mu$ and a filling $\sigma=\sigma_1 \sqcup \cdots \sqcup \sigma_p\in\T(\mu)$, define $\rev(\sigma)=\rev(\sigma_1)\sqcup\cdots\sqcup \rev(\sigma_p)$ as the reverse of $\sigma$, where the filling $\rev(\sigma_i)$ is obtained from $\sigma_i$ by reversing the sequence of entries in each row (see Figure \ref{F:b2}).
\end{definition}
We adopt some notations from \cite{AMM23,O:24} to describe the flip operator. 
An integer $i$ is said to be {\em $\mu$-compatible} if $\mu'_i=\mu'_{i+1}\ge 1$ and $1\le i<\mu_1$. That is, $i$
 is $\mu$-compatible if columns $i$ and $i+1$ are of equal height in the Young diagram of $\mu$.
 \begin{definition}[Flip operator $\rho_i^r$]\label{Def:2}
For a filling $\sigma\in\T(\mu)$, a $\mu$-compatible index $i$ and an integer $r$ such that $1\le r\le \mu_i'$, let $t_i^{(r)}$ be the operator that acts on $\sigma$ by interchanging the entries $\sigma(r,i)$ and $\sigma(r,i+1)$. For $1\le s\le r\le \mu'_i$, let
\begin{align*}
t_i^{[s,r]}:=t_i^{(s)}\circ t_i^{(s+1)}\cdots \circ t_i^{(r)}
\end{align*}
 denote the flip operator that swaps the entries of boxes $(x,i)$ and $(x,i+1)$ for all $x$ with $s\le x\le r$. 
 The flip operator $\rho_i^r$ is defined as follows:
 \begin{itemize}
 	\item  If $\sigma(x,i)=\sigma(x,i+1)$ for all $1\le x\le r$, then $\rho_i^r(\sigma)=\sigma$.
 	\item Otherwise, let $k$ be the largest integer such that $\sigma(k,i)\ne \sigma(k,i+1)$ and $k\le r$. Let $h$ be the largest integer satisfying $h\le k$, $\sigma(h,i)\ne \sigma(h,i+1)$ and
 	 \begin{align*}
 		\CMcal{Q}(\sigma(h,i),\sigma(h-1,i),\sigma(h-1,i+1))
 		=\CMcal{Q}(\sigma(h,i+1),\sigma(h-1,i),\sigma(h-1,i+1)),
 	\end{align*}
 where we set $\sigma(0,j)=\infty$ for $1\le j\le \mu_1$. We define $\rho_i^r(\sigma)=t_i^{[h,k]}(\sigma)$.
 \end{itemize}
 We call rows $k$ and $h$ the starting and terminating rows of $\rho_i^{r}$, respectively. For simplicity, we write
 \begin{align*}
 \rho_i=\rho_i^{\mu'_i}\quad\mbox{ and }\quad t_i=t_i^{(\mu'_i)}.
 \end{align*}
 By definition, $\rho_i^r\circ\rho_i^r(\sigma)=\sigma$; that is, $\rho_i^r$ is an involution on $\T(\mu)$.
\end{definition}
\begin{remark}
We point out two differences between the flip operator in Definition \ref{Def:2} and
the queue inversion flip operator that appears in \cite[Definition 7.3]{AMM23} and \cite[Definition 2.18]{O:24}. First, the operator in \cite[Definition 7.3]{AMM23} always starts from the topmost row (that is, $\rho_i$), whereas our flip operator is allowed to begin at any row. Second, we add the condition $\sigma(h,i)\ne \sigma(h,i+1)$ to the terminating row $h$. These two modifications are intended to quantify the changes in $\quinv$, $\inv$ and ${N}\des$ (see Theorem \ref{thm3.4}), which ultimately contribute to the proof of Theorem \ref{T:3}.
\end{remark}
\begin{example}
	Let $\mu=(2^6)$. For the filling $\sigma\in \T(\mu)$, the $\mu$-compatible index $i=1$ and $r=\mu_1'=6$, we show that $\rho_1(\sigma)=\rho_1^6(\sigma)=t_1^{[3,5]}(\sigma)$. 
	
	Since $\raisebox{-2pt}{\young(93)}$ is the topmost row with different entries, the flipping process
	$\rho_1$ begins at this row; that is, $k=5$. We next determine $h$.  
		Since $\CMcal{Q}(9,2,5)\ne\CMcal{Q}(3,2,5)$ and $\CMcal{Q}(2,3,9)\ne\CMcal{Q}(5,3,9)$, while $\CMcal{Q}(3,5,8)=\CMcal{Q}(9,5,8)$ and $3\ne 9$, 
	    we obtain $h=3$ by Definition \ref{Def:2}. Equivalently, the flipping process
		$\rho_1$ terminates in row $3$. Consequently, $\rho_1(\sigma)=t_1^{[3,5]}(\sigma)$.
	\begin{align*}
	\sigma=\,\,\,\begin{ytableau}[] 3&3\\ *(green) 9 & *(green) 3\\ 2 & 5 \\ 3 & 9  \\ 5 & 8  \\ 9 & 6
	\end{ytableau}
	\rightarrow
	\begin{ytableau}[] 3&3\\ 3 & 9\\ *(green) 2 & *(green) 5 \\ 3 & 9  \\ 5 & 8  \\ 9 & 6
	\end{ytableau}
	\rightarrow
	\begin{ytableau}[] 3&3\\ 3 & 9\\ 5 & 2 \\ *(green) 3 & *(green) 9  \\ 5 & 8  \\ 9 & 6
	\end{ytableau}
	\rightarrow
	\begin{ytableau}[] 3&3\\ 3 & 9\\ 5 & 2 \\ 9 & 3  \\ 5 & 8  \\ 9 & 6
	\end{ytableau}\,\,\,=\rho_1(\sigma)
	\end{align*}
	
\end{example}

\section{Roadmap for the proof of Theorem \ref{T:3}}\label{S:mainproof}
The purpose of this section is to present our strategy to prove Theorem \ref{T:3}. The proof is bijective: for a partition $\mu$, we construct a bijection $\varphi:\T(\mu)\rightarrow \T(\mu)$ satisfying $\varphi(\sigma)\sim \sigma$ and 
\begin{align}\label{E:varphi1}
(\quinv,\maj)(\varphi(\sigma))=(\inv,\maj)(\sigma).
\end{align}
In particular, if the Young diagram of $\mu$ is a rectangle, we prove that $$(\inv,\quinv,\maj)(\varphi(\sigma))=(\quinv,\inv,\maj)(\sigma).$$

The bijection $\varphi$ is a composition of two bijections associated with the flip operator $\rho_i^{r}$. The first one $\gamma$ is described in Theorem \ref{pro1} below, and reduces to the reverse operator when the Young diagram of $\mu$ is rectangular. Let
$\sigma=\sigma_1\sqcup\cdots \sqcup\sigma_p$, and define
\begin{align}\label{eqthm3.6a}
\kappa(\sigma)= \sum_{i=1}^{p}(\quinv(\sigma_i)-\inv(\rev(\sigma_i))).
\end{align}
For any filling $\tau$, $\tau_1$ denotes the leftmost rectangular filling in the decomposition of $\tau$.
\begin{theorem}\label{pro1}
	There is a bijection $\gamma:\T(\mu)\rightarrow \T(\mu)$ satisfying $\gamma(\sigma)=\tau\sim \sigma$,
	\begin{align}\label{eqthm3.6b}
	\quinv(\tau)&=\inv(\sigma)+\kappa(\tau),\\
	\label{eqthm3.6c}
	\maj(\tau)&=\maj(\sigma),\\
	\label{eqthm3.6d}{N}\des(\sigma_1)&={N}\des(\rev(\tau_1)),
	\end{align}
	and the topmost rows of $\sigma$ and $\tau$ are the reverses of each other. 
\end{theorem}
The second bijection $\theta:\T(\mu)\rightarrow \T(\mu)$ acts on each rectangle of any filling independently and decreases the number of queue inversion triples by $\kappa(\gamma(\sigma))$ while preserving the major index. Using this, we find the bijection $\varphi$ with the property (\ref{E:varphi1}). Because both bijections $\gamma$ and $\theta$ are constructed using the involution $\phi_i$ defined below, we first prove Theorem \ref{thm3.4} in Section \ref{S:5}, and then establish Theorems \ref{T:3} and \ref{pro1} in Section \ref{S:6}.
\begin{theorem}\label{thm3.4}
	For a partition $\mu$ and a $\mu$-compatible index $i$, let $\sigma\in\T(\mu)$ and let $x_i$ be the number of non-descents in the $i$th column of $\sigma$.	
	Then there is an involution $\phi_i:\T(\mu)\rightarrow \T(\mu)$ such that $\phi_i(\sigma) \sim \sigma$, and for $\nu\in\{\inv,\quinv\}$,
	\begin{align}
	\label{E:majthm62}\maj(\phi_i(\sigma))&= \maj(\sigma), \\
	\label{E:quinvthm62}\nu(\phi_i(\sigma))&=\nu(\sigma)+x_{i+1}-x_i,\\
	\label{E:ndesthm62}{N}\des(\phi_i(\sigma))&=s_i\,{N}\des(\sigma),
	\end{align}
	where ${s_i}\,(\ldots x_i,x_{i+1}\ldots)=(\ldots x_{i+1},x_{i}\ldots)$.
\end{theorem}

\section{Involution on arbitrary fillings}\label{S:5}
This section contains a sequence of auxiliary lemmas that head to Theorem \ref{thm3.4}. We begin by discussing how the statistics $\quinv$, $\inv$ and $\maj$ change under the reverse and flip operators; then, we proceed by exploiting these properties to confirm Theorem \ref{thm3.4}.
\subsection{Operators and statistics}
Let $\sigma\vert_{i}^{j}$ denote the segment of $\sigma$ from row $i$ through row $j$. For any filling $\sigma$ of a rectangular diagram, we can express $\kappa(\sigma)$ explicitly. For simplicity, define $\chi(\A)=1$ if  statement $\A$ is true, and $\chi(\A)=0$ otherwise.
\begin{lemma}\label{L:1}
	For a partition $\mu=(n^m)$ and a filling $\sigma\in\T(\mu)$, let $x_i$ be the number of non-descents in column $i$ of $\sigma$. Then, 
	\begin{align}\label{eq5}
	\quinv(\sigma)-\inv(\rev(\sigma))=\inv(\sigma)-\quinv(\rev(\sigma))
	=\sum_{i=1}^{n} x_i(n-2i+1).
	\end{align}
\end{lemma}
\begin{proof}
Since (\ref{eq5}) is trivially true for $m=1$ or $n=1$, we assume $m\ge 2$ and $n\ge 2$.	
Given a partition $\mu=(n^m)$ and a filling $\sigma\in \T(\mu)$, we have, by definition,  
\begin{align}\label{E:nu1}
\nu(\sigma)&=\nu(\sigma\vert_{1}^{m-1})+\nu(\sigma\vert_{m-1}^{m})-\nu(\sigma\vert_{m-1}^{m-1})
\end{align}
for $\nu\in\{\inv,\quinv\}$, which eventually leads to
\begin{align}\label{E:inv1}
\nu(\sigma)&
=\sum_{i=1}^{m-1}\nu(\sigma\vert_{i}^{i+1})-\sum_{i=2}^{m-1}\nu(\sigma\vert_{i}^{i}).
\end{align}
Considering $\inv(\sigma\vert_{i}^{i})=\quinv(\rev(\sigma)\vert_{i}^{i})$, we obtain
\begin{align}\label{E:inv_quinv1}
\inv(\sigma)-\quinv(\rev(\sigma))=\sum_{i=1}^{m-1}(\inv(\sigma\vert_{i}^{i+1})-\quinv(\rev(\sigma)\vert_{i}^{i+1})),
\end{align}
which reduces the task of counting $\inv(\sigma)-\quinv(\rev(\sigma))$ to counting this quantity for the case of a two-row rectangular filling. That is, the case $m=2$.

Assume $\mu=(n^2)$ and let $\tau=\begin{ytableau}[] a_1\\ b_1
\end{ytableau}
\begin{ytableau}[] a_2\\ b_2
\end{ytableau}
\begin{ytableau} \cdots\\ \cdots
\end{ytableau}
\begin{ytableau}[] a_n\\ b_n
\end{ytableau}\in \T(\mu)$. Then, 
we calculate the difference
\begin{align}
&\qquad\inv(\tau)-\quinv(\rev(\tau))\notag\\
\label{E:inq1}&=\sum_{1\le i<j\le n}\CMcal{Q}(a_i,b_i,a_j)+\CMcal{Q}(b_i,\infty,b_j)-\CMcal{Q}(a_j,b_j,b_i)-\CMcal{Q}(0,a_j,a_i)
\end{align}
by distinguishing whether $(a_i,b_i)$ or $(a_j,b_j)$ is a descent of $\tau$.
\begin{enumerate}
\item
If $a_i>b_i$, then $\CMcal{Q}(a_i,b_i,a_j)+\CMcal{Q}(b_i,\infty,b_j)
=\chi(b_i<a_j<a_i)+\chi(b_i>b_j)$; otherwise, $a_i\le b_i$ and $\CMcal{Q}(a_i,b_i,a_j)+\CMcal{Q}(b_i,\infty,b_j)=\chi(a_j<a_i)+\chi(b_i<a_j)+\chi(b_i>b_j)$.
\item
If $a_j>b_j$, then $\CMcal{Q}(a_j,b_j,b_i)+\CMcal{Q}(0,a_j,a_i)=\chi(b_j<b_i<a_j)+\chi(a_j<a_i)$; otherwise, $a_j\le b_j$ and $\CMcal{Q}(a_j,b_j,b_i)+\CMcal{Q}(0,a_j,a_i)=\chi(a_j<a_i)+\chi(b_i<a_j)+\chi(b_i>b_j)$.
\end{enumerate}
It is easily seen that the summand on the right-hand side of (\ref{E:inq1}) equals zero if both $(a_i,b_i)$ and $(a_j,b_j)$ are non-descents of $\tau$. If both $(a_i,b_i)$ and $(a_j,b_j)$ are descents of $\tau$, then 
\begin{align*}
&\quad\CMcal{Q}(a_i,b_i,a_j)+\CMcal{Q}(b_i,\infty,b_j)-\CMcal{Q}(a_j,b_j,b_i)-\CMcal{Q}(0,a_j,a_i),\\
&=\chi(b_i<a_j<a_i)-\chi(a_j<a_i)+\chi(b_i>b_j)-\chi(b_j<b_i<a_j),\\
&=\chi(b_j<b_i\,\,\mbox{and}\,\, b_i\ge a_j)-\chi(a_j<a_i\,\,\mbox{and}\,\, b_i\ge a_j)=0.
\end{align*}
The last equality follows by discussing the cases $b_i<a_j$ and $b_i\ge a_j$. If $b_i<a_j$, then clearly $\chi(b_j<b_i\,\,\mbox{and}\,\, b_i\ge a_j)-\chi(a_j<a_i\,\,\mbox{and}\,\, b_i\ge a_j)=0$; otherwise, if $b_i\ge a_j$, then $a_i>b_i\ge a_j>b_j$, implying that $\chi(b_j<b_i\,\,\mbox{and}\,\, b_i\ge a_j)-\chi(a_j<a_i\,\,\mbox{and}\,\, b_i\ge a_j)=0$.

If $(a_i,b_i)$ is not a descent while $(a_j,b_j)$ is a descent, then
\begin{align}
&\quad\, \nonumber\CMcal{Q}(a_i,b_i,a_j)+\CMcal{Q}(b_i,\infty,b_j)-\CMcal{Q}(a_j,b_j,b_i)-\CMcal{Q}(0,a_j,a_i)\\
\nonumber&=\chi(b_i<a_j)+\chi(b_i>b_j)-\chi(b_j<b_i<a_j)\\
\nonumber &=\chi(b_i<a_j\,\,\mbox{and}\,\, b_i\le b_j)+\chi(b_i>b_j)\\
\label{E:inc1}&=\chi(b_i\le b_j)+\chi(b_i>b_j)=1.
\end{align}
The second-to-last equality holds because $b_i<a_j$ follows from $b_i\le b_j$ under the condition $a_j>b_j$. Similarly, if $(a_i,b_i)$ is a descent and $(a_j,b_j)$ is not a descent, then we have
\begin{align}
&\quad\, \nonumber\CMcal{Q}(a_i,b_i,a_j)+\CMcal{Q}(b_i,\infty,b_j)-\CMcal{Q}(a_j,b_j,b_i)-\CMcal{Q}(0,a_j,a_i)\\
\nonumber&=\chi(b_i<a_j<a_i)-\chi(a_j<a_i)-\chi(b_i<a_j)\\
\label{E:dec1}&=-\chi(b_i\ge a_j)-\chi(b_i<a_j)=-1.
\end{align}
Substituting (\ref{E:inc1}) and (\ref{E:dec1}) back into (\ref{E:inq1}) gives 
\begin{align*}
\inv(\tau)-\quinv(\rev(\tau))&=
|\{(i,j):a_i\le b_i\,\, \mbox{and}\,\,a_j>b_j,\,1\le i<j\le n\}|\\
  &\quad\quad-|\{(i,j):a_i>b_i\,\,\mbox{and}\,\,a_j\le b_j,\,1\le i<j\le n\}|,\\
&=|\{(i,j):a_i\le b_i,\,1\le i<j\le n\}|
-|\{(i,j):a_j\le b_j,\,1\le i<j\le n\}|\\
&=\sum_{i=1}^{n}\chi(a_i\le b_i)(n-2i+1).
\end{align*}
Therefore, the second equality in (\ref{eq5}) holds upon combining with (\ref{E:inv_quinv1}). An equivalent form of this equality is
\begin{align*}
\quinv(\sigma)-\inv(\rev(\sigma))
=\sum_{i=1}^{n} x_i(n-2i+1),
\end{align*}
which is obtained by substituting $i$ with $n-i+1$ in (\ref{eq5}). Therefore, the proof is complete.
\end{proof}

Let us now review the central properties of the flip operators.
\begin{lemma}\label{lem3.1}\cite[Lemmas 7.5 and 7.6]{AMM23}
	For a partition $\mu$ and a $\mu$-compatible $i$, let $\sigma\in\T(\mu)$, where $a,b$ are the entries in the top row of columns $i$ and $i+1$, respectively. If $a\ne b$, then
	\begin{align}\label{eq3.1}
	\maj\left(\rho_{i}(\sigma)\right)&=\operatorname{maj}(\sigma),\\
	\label{eq3.2}
	\quinv\left(\rho_{i}(\sigma)\right)&=\quinv(\sigma)+\chi(a>b)-\chi(a<b).
	\end{align}
Moreover, the number of queue inversion triples located between columns $i$ or $i+1$ and column $j$ of $\sigma$ for any $j>i+1$ is invariant under $\rho_i$.
\end{lemma}
We extend Lemma \ref{lem3.1} by analyzing the operator $\rho_i^r$ that initiates from any row $r$ (see Lemma \ref{L:2}). Although the proof of Lemma \ref{L:2} is analogous to that of Lemma \ref{lem3.1}, we provide it to make this paper self-contained. 

For a sub-filling $\nu$ of a filling $\sigma$, let $f(\nu)$ denote the filling obtained from $\nu$ after applying $f$ to $\sigma$. 


\begin{lemma}\label{L:2}
For a partition $\mu$ and a $\mu$-compatible index $i$,
let $\sigma\in\T(\mu)$ and suppose that 
$$\rho_i^{r}(\sigma)=t_i^{[\kappa_1,\kappa_2]}(\sigma).$$ 
Let $\raisebox{-8pt}{\young(ab,cd)}$ and $\raisebox{-8pt}{\young(st,uv)}$ be parts of $\sigma$ between columns $i$ and $i+1$ such that $\raisebox{-2pt}{\young(cd)}$ corresponds to the starting row $\kappa_2$ and $\raisebox{-2pt}{\young(st)}$ corresponds to the terminating row $\kappa_1$. We set $\sigma(0,j)=\infty$ and $\sigma(\mu_j'+1,j)=0$ for $1\le j\le \mu_1$.
\begin{itemize}
\item If $\CMcal{Q}(a,c,d)=\CMcal{Q}(b,c,d)=0$, then
\begin{align}\label{E:quinv1}
\quinv(\sigma)+1=\quinv(\rho_i^{r}(\sigma)).
\end{align}
Equivalently, if $\CMcal{Q}(a,c,d)=\CMcal{Q}(b,c,d)=1$, then
\begin{align}\label{E:quinv2}
	\quinv(\sigma)-1=\quinv(\rho_i^{r}(\sigma)).
\end{align}
\item  If $\CMcal{Q}(a,c,d)=\CMcal{Q}(b,c,d)$ and $\CMcal{Q}(s,u,t)=\CMcal{Q}(s,v,t)=0$, then
\begin{align}\label{E:inv11}
	\inv(\sigma)+1=\inv(\rho_i^{r}(\sigma)).
\end{align}
Equivalently, if $\CMcal{Q}(a,c,d)=\CMcal{Q}(b,c,d)$ and $\CMcal{Q}(s,u,t)=\CMcal{Q}(s,v,t)=1$, then
\begin{align}\label{E:inv12}
	\inv(\sigma)-1=\inv(\rho_i^{r}(\sigma)).
\end{align}
\end{itemize}
For all cases, that is, whenever $\CMcal{Q}(a,c,d)=\CMcal{Q}(b,c,d)$, we have
\begin{align}\label{E:maj1}
	\maj(\sigma)=\maj(\rho_i^{r}(\sigma)),
\end{align}
and the number of queue inversion triples (resp. inversion triples) induced between columns $i$ or $i+1$ and column $j$ of $\sigma$ for any $j>i+1$ is invariant under $\rho_i^r$.
\end{lemma}
\begin{remark}
	An important observation is that 
	$\CMcal{Q}(s,u,t)=\CMcal{Q}(s,v,t)$ if and only if $\CMcal{Q}(s,u,v)=\CMcal{Q}(t,u,v)$ for arbitrary four entries $s,t,u$ and $v$. Since $\raisebox{-2pt}{\young(st)}$ is the terminating row of $\rho_i^r$, we have $\CMcal{Q}(s,u,v)=\CMcal{Q}(t,u,v)$. It follows that $\CMcal{Q}(s,u,t)=\CMcal{Q}(s,v,t)$, implying that either (\ref{E:inv11}) or (\ref{E:inv12}) holds provided that $\CMcal{Q}(a,c,d)=\CMcal{Q}(b,c,d)$.
\end{remark}
\begin{proof}
As the proofs of (\ref{E:inv11}) and (\ref{E:inv12}) are along the lines of (\ref{E:quinv1}) and (\ref{E:quinv2}), respectively, we only prove the latter. 

First, we identify the equivalence between (\ref{E:quinv1}) and (\ref{E:quinv2}). Let $\rho_i^{r}(\sigma)=\tau$. Then, $\sigma=\rho_i^{r}(\tau)$ because $\rho_i^{r}$ is an involution. Whether $c<d$ or $c>d$, we have $\CMcal{Q}(a,c,d)=\CMcal{Q}(b,c,d)=0$ if and only if  $\CMcal{Q}(a,d,c)=\CMcal{Q}(b,d,c)=1$. It follows that (\ref{E:quinv1}) and (\ref{E:quinv2}) are equivalent. Therefore, it suffices to prove (\ref{E:quinv1}) and (\ref{E:maj1}) when $\CMcal{Q}(a,c,d)=\CMcal{Q}(b,c,d)$.

Changes in $\maj$ are only possible on the borders; namely, between rows $\kappa_2+1$ and $\kappa_2$, and between rows $\kappa_1-1$ and $\kappa_1$. Here, we only examine the former, because the latter is discussed similarly.
 
 Let $\pi=\raisebox{-8pt}{\young(ab,cd)}$. Then $\rho_i^{r}(\pi)=\raisebox{-8pt}{\young(ab,dc)}$, and we verify that $\maj(\pi)=\maj(\rho_i^{r}(\pi))$ when $\CMcal{Q}(a,c,d)=\CMcal{Q}(b,c,d)$.

If $\CMcal{Q}(a,c,d)=\CMcal{Q}(b,c,d)=0$, then either $c<\mini(a,b)\le \maxi(a,b)\le d$, or $d<c<\mini(a,b)$, or $\maxi(a,b)\le d<c$, or $\mini(a,b)\le d<c<\maxi(a,b)$. Otherwise, if $\CMcal{Q}(a,c,d)=\CMcal{Q}(b,c,d)=1$, then either $\maxi(a,b)\le c<d$, or  $c<d<\mini(a,b)$, or $d<\mini(a,b)\le \maxi(a,b)\le c$ or $\mini(a,b)\le c<d<\maxi(a,b)$. In all scenarios, we have $\maj(\pi)=\maj(\rho_i^{r}(\pi))$, resulting in (\ref{E:maj1}).

It remains to prove (\ref{E:quinv1}). 
First, we assert that $\rho_i^r$ preserves the number of queue inversion triples induced between the columns $i$ and $i+1$ that lie strictly below row $\kappa_2+1$. 

Let $\pi=\raisebox{-8pt}{\young(\alpha\beta,\gamma\delta)}$. If both rows of $\pi$ are reversed by $\rho_i^r$, then we must have $\delta\ne \gamma$ and $\CMcal{Q}(\alpha,\gamma,\delta)\ne \CMcal{Q}(\beta,\gamma,\delta)$. Since $\CMcal{Q}(\beta,\gamma,\delta)\ne \CMcal{Q}(\beta,\delta,\gamma)$ for $\delta\ne \gamma$, it follows that  $\CMcal{Q}(\alpha,\gamma,\delta)=\CMcal{Q}(\beta,\delta,\gamma)$. If $\raisebox{-2pt}{\young(\alpha\beta)}=\raisebox{-2pt}{\young(st)}$ is the terminating row $\kappa_1$, then $\CMcal{Q}(\alpha,\gamma,\delta)=\CMcal{Q}(\beta,\gamma,\delta)$ by definition. In both cases, the numbers of queue inversion triples derived from $\pi$ and $\rho_i^r(\pi)$ are equal. This proves the assertion.

If $\raisebox{-2pt}{\young(\gamma\delta)}=\raisebox{-2pt}{\young(cd)}$ is the starting row $\kappa_2$, then $c\ne d$. Since
$\CMcal{Q}(a,c,d)=0$ and $c\ne d$, we find
$\CMcal{Q}(a,d,c)=1$, which yields 
\begin{align}\label{E:qupi}
	\quinv(\pi)+1=\quinv(\rho_i^{r}(\pi)).
\end{align}
Second, we claim that the number of queue inversion triples induced by columns $i$ or $i+1$ and column $j$ for all $j>i+1$ is invariant under $\rho_i^r$. That is, this number is preserved by both $t_i^{(\kappa_2)}$ and $t_i^{(\kappa_1)}$. Here, we prove the claim only for $t_i^{(\kappa_2)}$ and omit the other analogous case. Consider the square of entries $a,b,c,d$ and let $z=\sigma(\kappa_2,j)$, as shown below.
\begin{center}
	\begin{minipage}[H]{0.2\linewidth}
		\begin{ytableau}
			a  & b & \none[]
			\\
			c & d & \none[\dots] & z\\
		\end{ytableau}
	\end{minipage}
	\begin{minipage}[H]{0.1\linewidth}
		\begin{ytableau}
				a  & b & \none[]
			\\
			d & c & \none[\dots] & z\\
		\end{ytableau}
	\end{minipage}
\end{center}
We examine all possible total orderings of $a,b,c,d$ such that $\CMcal{Q}(a,c,d)=\CMcal{Q}(b,c,d)=0$; namely:  $c>d\geq a>b$, $c>d\ge b\geq a$, $b>c>d\geq a$, $d\geq a>b>c$, $d\geq b\geq a>c$, $a>c>d\ge b$, $a>b>c>d$ and $b\geq a>c>d$.
For each of these situations, five possible ways exist to insert $z$, as outlined in Table \ref{Tab:1}. Note that only the cases when $d\geq a>b>c$ or $d\geq b\geq a>c$ are listed, because the others follow in the same manner. One readily sees that $\CMcal{Q}(a,c,z)+\CMcal{Q}(b,d,z)=\CMcal{Q}(a,d,z)+\CMcal{Q}(b,c,z)$. That is, the number of queue inversion triples is preserved in this case, which, in combination with the previous assertion and (\ref{E:qupi}), proves (\ref{E:quinv1}).

\begin{table}
\begin{center}
	\begin{tabular}{ccccccc}
		&  & $(\CMcal{Q}(a,c,z),\CMcal{Q}(b,d,z))$ &$(\CMcal{Q}(a,d,z),\CMcal{Q}(b,c,z))$\\
		\hline \\[-1.5ex]
		$z>d\geq a>b>c$ &  & $(0,1)$ & $(1,0)$\\[1ex]
		$z>d\geq b\geq a>c$& & $(0,1)$ & $(1,0)$\\
		\hline \\[-1.5ex]
		$d\geq z\geq a>b>c$ &  & $(0,0)$ & $(0,0)$ \\[1ex]
		$d\geq z\geq b\geq a>c$ & & $(0,0)$ & $(0,0)$ \\
		\hline \\[-1.5ex]
		$d\geq a>z\geq b>c$  &  & $(1,0)$ & $(1,0)$ \\[1ex]
		$d\geq b>z\geq a>c$ &   & $(0,1)$ & $(0,1)$\\
		\hline \\[-1.5ex]
		$d\geq a>b>z>c$  & & $(1,1)$ & $(1,1)$ \\[1ex]
		$d\geq b\geq a>z>c$ & & $(1,1)$ & $(1,1)$\\
		\hline \\[-1.5ex]
		$d\geq a>b>c\geq z$  & & $(0,1)$ & $(1,0)$ \\[1ex]
		$d\geq b\geq a>c\geq z$ & & $(0,1)$ & $(1,0)$ \\[1ex]
	\end{tabular}
\caption{Change in queue inversion triples for the cases when $d\geq a>b>c$ or $d\geq b\geq a>c$.} \label{Tab:1}
\end{center}
\end{table}
\end{proof}

\subsection{Proof of Theorem \ref{thm3.4}}
\
\vspace{-5mm}
\newline

Let $\tau=\raisebox{-8pt}{\young(ab,cd)}$. If exactly one of the columns of $\tau$ forms a descent; that is, $\chi(a>c)\ne \chi(b>d)$, then we call $\tau$ a {\em descent block}. It is a {\em right-descent block}
or {\em left-descent block}, depending on
whether the right or left column forms a descent pair. Otherwise, both columns of $\tau$ form descent pairs, or they both form non-descent pairs (that is, $\chi(a>c)=\chi(b>d)$). In this case, we call $\tau$ a {\em neutral block}. 

Every rectangular filling $\sigma$ of shape $(2^m)$ can be decomposed vertically into a sequence of blocks
 $\tau_1,\ldots,\tau_{m-1}$ from top to bottom such that $\tau_i$ and $\tau_{i+1}$ overlap in exactly one row for $1\le i<m-1$. We write 
\begin{align}\label{E:odot}
\sigma=\tau_1\odot\cdots \odot \tau_{m-1}.
\end{align}
Two descent blocks $\tau_i$ and $\tau_j$ are called {\em compatible} if they are both left-descent blocks or both right-descent blocks. They are said to be {\em neighbours} if $j=i+1$ or if $i<j$ and all blocks $\tau_{i+1},\ldots,\tau_{j-1}$ between them are neutral blocks (see Example \ref{E5.1}).
\begin{example}\label{E5.1}
	Let $\sigma$ be the filling of shape $(2^6)$, as shown below.
	\begin{align*}
		\sigma=\begin{ytableau}[]8 & 6 \\
			4 & 6 \\
			6 & 7 \\
			3 & 8 \\
			6 & 5 \\
			5 & 2
		\end{ytableau}\,.
	\end{align*}
Using (\ref{E:odot}), we have $\sigma=\tau_1\odot\tau_2\odot\tau_3\odot\tau_4\odot \tau_{5}$, where  
$$\tau_1=\raisebox{-8pt}{\young(86,46)}\,,\,\tau_2=\raisebox{-8pt}{\young(46,67)}\,,\,\tau_3=\raisebox{-8pt}{\young(67,38)}\,,\,\tau_4=\raisebox{-8pt}{\young(38,65)}\,,\,\tau_5=\raisebox{-8pt}{\young(65,52)}\,.$$
	
Blocks $\tau_1$, $\tau_3$, and $\tau_4$ are descent blocks, whereas blocks $\tau_2$ and $\tau_5$ are neutral blocks. Here, blocks $\tau_1$ and $\tau_3$ are compatible neighbours because they are both left-descent ones, and the only block between them (block $\tau_2$) is neutral. Furthermore, blocks $\tau_3$ and $\tau_4$ are incompatible neighbours, because $\tau_3$ is a left-descent block, $\tau_4$ is a right-descent block, and no block lies between them.

\end{example}
We classify all descent blocks into three types.
\begin{definition}\label{Def:3}
	Let $\tau=\raisebox{-8pt}{\young(ab,cd)}$ be a descent block. We define
	\begin{itemize}
		\item $\tau\in\mathcal{A}$ if $c\ne d$ and $\{c,d\}$ is the set of the smallest and largest integers among  $a,b,c,d$; that is,
		$d\ge b\ge a>c$, $d\ge a>b>c$, $c\ge b\ge a>d$, or $c\ge a>b>d$.
		\item $\tau\in\mathcal{B}$ if $a\ne b$ and $\{a,b\}$ is the set of the smallest and largest integers of $a,b,c,d$; that is, $a>c\ge d\ge b$, $a>d>c\ge b$, $b>c\ge d\ge a$, or $b>d>c\ge a$.
		\item $\tau\in\mathcal{C}$ if $a>d\ge b>c$, $d\ge a>c\ge b$, $b>c\ge a>d$, or $c\ge b>d\ge a$.
	\end{itemize}
\end{definition}

\begin{remark}
Equivalently, $\tau\in\A$ if $\maxi(c,d)\ge \maxi(a,b)$ and $\mini(c,d)<\mini(a,b)$; $\tau\in \B$ if
$\maxi(c,d)<\maxi(a,b)$ and $\mini(c,d)\ge \mini(a,b)$; $\tau\in \mathcal{C}$ if one of the following is true:
\begin{itemize}
	\item $\maxi(c,d)\ge \maxi(a,b)$, $\mini(c,d)\ge \mini(a,b)$ and $\chi(a>c)\ne \chi(b>d)$;
	\item $\maxi(c,d)<\maxi(a,b)$, $\mini(c,d)<\mini(a,b)$ and $\chi(a>c)\ne \chi(b>d)$.
\end{itemize}
This exhausts all possible total orderings of $a$, $b$, $c$ and $d$ when they are subject to the condition that $\chi(a>c)\ne \chi(b>d)$.
\end{remark}
\begin{remark}\label{Rem:1}
For $\tau\in\A \cup \B$, we have $\CMcal{Q}(a,c,d)=\CMcal{Q}(b,c,d)$, whereas for $\tau\in\mathcal{C}$, $\CMcal{Q}(a,c,d)\ne \CMcal{Q}(b,c,d)$. Following the notations in Example \ref{E5.1}, $\tau_1\in\mathcal{C}$, $\tau_3\in \A$ and $\tau_4\in \B$. 
\end{remark}

We define a map $\varepsilon$ that associates a descent block with a flip operator, as follows.

\begin{definition}\label{Def:varepsilon}
For any partition $\mu$ and a $\mu$-compatible index $i$, let $\sigma\in\T(\mu)$ 
and set $\sigma(\mu_j'+1,j)=0$ and $\sigma(0,j)=\infty$ for all $1\le j\le \mu_1$. Let $\tau=\raisebox{-8pt}{\young(ab,cd)}$ be a descent block of $\sigma$ such that $c=\sigma(r,i)$.

If $\tau\in\A$, then define $\varepsilon(\tau)=\rho_i^{r}$; otherwise define $\varepsilon(\tau)=\rho_i^{\kappa}$, where $\kappa$ is the smallest integer $x$ such that $x>r$ and
\begin{align}\label{E:kappa}
\CMcal{Q}(\sigma(x+1,i),\sigma(x,i),\sigma(x,i+1))=\CMcal{Q}(\sigma(x+1,i+1),\sigma(x,i),\sigma(x,i+1))=\chi(a\le c).
\end{align}
\end{definition}
Since $\rho_i^r$ is well-defined for any $r$, we only need to establish the existence of $\kappa$ for $\tau\not\in \A$ to show that $\varepsilon$ is well-defined on $\tau$. 
We say that a block $\rho$ is {\em above} another block $\tau$ if $\rho$ and $\tau$ are located in the same columns and the top row of $\rho$ is above the top row of $\tau$. 

\begin{lemma}\label{L:lemma10}
Following the notations of Definition \ref{Def:varepsilon}, the map $\varepsilon$ is well-defined on $\tau$ provided that one of (i) and (ii) holds: (i) all blocks above $\tau$ are neutral blocks; (ii) the neighbour of $\tau$ from above belongs to $\B$, or it is compatible with $\tau$.
\begin{enumerate}
\item Under conditions $(i)$ and $\tau\not\in \A$, we have $\sigma(j,i)>\sigma(j,i+1)$ for all $r<j\le \kappa$ when $a>c$, and $\sigma(j,i)<\sigma(j,i+1)$ when $a\le c$.
\item
Under condition $(ii)$, let $\rho$ denote the neighbour of $\tau$ from above. Then, the terminating row of $\varepsilon(\rho)$ is above the starting row of $\varepsilon(\tau)$.
\end{enumerate}
\end{lemma}
\begin{proof}
	Without loss of generality, we assume that $a>c$ and $b\le d$. It suffices to consider the case $\tau\in\B\cup \mathcal{C}$ for $(i)$ and $(ii)$. 
	For $(i)$, suppose that $\kappa$ does not exist for $\varepsilon(\tau)$. Then, at least one of $\CMcal{Q}(\sigma(j+1,i),\sigma(j,i),\sigma(j,i+1))$ and $\CMcal{Q}(\sigma(j+1,i+1),\sigma(j,i),\sigma(j,i+1))$ equals $1$ for all $j>r$. In this case, we show by induction on $j$ that $\sigma(j,i)>\sigma(j,i+1)$ for all $j>r$.
	
	 The base case $j=r+1$ holds, since $a>b$. If $a\le b$, then $c<a\le b\le d$, contradicting the assumption that $\tau\not\in\A$. Assuming that $\sigma(x,i)>\sigma(x,i+1)$ for all $r<x\le j-1$, we now prove the claim for $x=j$.
	
	We claim that exactly one of $\CMcal{Q}(\sigma(j,i),\sigma(j-1,i),\sigma(j-1,i+1))$ and $\CMcal{Q}(\sigma(j,i+1),\sigma(j-1,i),\sigma(j-1,i+1))$ equals $1$. Suppose otherwise; that is, both equal $1$. This implies $\sigma(j-1,i)\ge \sigma(j,i)>\sigma(j-1,i+1)$ and $\sigma(j-1,i)\ge \sigma(j,i+1)>\sigma(j-1,i+1)$. It follows that the block formed by $\sigma(j-1,i)$, $ \sigma(j,i)$, $ \sigma(j-1,i+1)$ and $\sigma(j,i+1)$ is a descent block, which contradicts condition $(i)$. This proves the claim.
	
	As a result, $\CMcal{Q}(\sigma(j,i),\sigma(j-1,i),\sigma(j-1,i+1))\ne \CMcal{Q}(\sigma(j,i+1),\sigma(j-1,i),\sigma(j-1,i+1))$. Since the block formed by $\sigma(j,i)$, $\sigma(j,i+1)$, $\sigma(j-1,i)$ and $\sigma(j-1,i+1)$ is a neutral block by $(i)$, and $\sigma(j-1,i)>\sigma(j-1,i+1)$ by the induction hypothesis, we conclude that $\sigma(j,i)>\sigma(j,i+1)$, completing the proof that $\sigma(j,i)>\sigma(j,i+1)$ for all $j>r$.
	
	Taking $j=\mu_i'$, we obtain $\sigma(\mu_i',i)>\sigma(\mu_i',i+1)$. However, $\CMcal{Q}(0,\sigma(\mu_i',i),\sigma(\mu_i',i+1))=0$, which implies that $\kappa=\mu_i'$, contradicting the assumed non-existence of $\kappa$. In consequence, $\kappa$ exists and $\varepsilon(\tau)$ is well-defined for $(i)$ when $a>c$. Similarly, one can prove the statement for $a\le c$, in which case
	$\sigma(j,i)<\sigma(j,i+1)$ for $r<j\le \kappa$. Thus, the proof of $(1)$ is complete.
		
	For $(ii)$, suppose that $\rho=\raisebox{-8pt}{\young(\alpha\beta,\gamma\delta)}$ is compatible with $\tau$ or $\rho\in \B$. Let row $k_1$ be the terminating row of $\varepsilon(\rho)$, and let row $k_2$ be the starting row of  $\varepsilon(\tau)$. We show that $k_2$ exists and that $k_1>k_2$. Suppose that $k_2$ does not exist. Then, $\gamma>\delta$ by $(1)$. If $\rho$ is compatible with $\tau$; that is, $\alpha>\gamma$ and $\beta\le \delta$, then $\alpha>\gamma>\delta\ge \beta$ in view of $\gamma>\delta$. Otherwise, $\alpha\le \gamma$, $\beta>\delta$ and $\rho\in \B$. By the definition of $\B$, we have $\beta>\gamma\ge \delta\ge \alpha$. In both cases,  we obtain $\CMcal{Q}(\alpha,\gamma,\delta)=\CMcal{Q}(\beta,\gamma,\delta)=0$, contradicting the assumption that $k_2$ does not exist. It follows that $k_2$ exists, and hence $\varepsilon(\tau)$ is well-defined. 
	
	We now discuss the type of $\tau$ to prove $k_1>k_2$. Let $c=\sigma(r,i)$ and $\gamma=\sigma(j,i)$, where  $j\ge r+1$ since $\rho$ is above $\tau$. 
	If $\tau\in\A$, then $k_2=r$ and $k_1\ge r+1$ as $\CMcal{Q}(a,c,d)=\CMcal{Q}(b,c,d)$. Thus, $k_1>k_2$.
	If $\tau\in \B\cup \mathcal{C}$, then $k_2>r$, and we claim that $k_2\le j$. Suppose otherwise; that is, $k_2>j$. Then, $\CMcal{Q}(\alpha,\gamma,\delta)=\CMcal{Q}(\beta,\gamma,\delta)=0$ by the preceding argument, contradicting the assumption that $k_2>j$. Therefore, we must have $k_2\le j$, and this further leads to $k_2<k_1$ by the definition of $k_1$. Consequently, statement (2) is proved.

	
\end{proof}
\begin{example}\label{Example:5.3}
	For the partition $\mu=(2^6)$, the $\mu$-compatible index $i=1$ and the filling $\sigma$ in Example  \ref{E5.1}, we have $\varepsilon(\tau_3)=\rho_1^{3}=t_1^{(3)}$, $\varepsilon(\tau_1)=\rho_1=t_1^{[5,6]}$ 
	and $\varepsilon(\tau_4)=\rho_1^4=t_1^{(4)}$ by Definition \ref{Def:varepsilon}. 
	
	The map $\varepsilon$ is well-defined on $\tau_1$ because no block is located above $\tau_1$, which is a special case of $(i)$. By Lemma \ref{L:lemma10}, $\tau_1$ has property $(1)$; namely, $\sigma(6,1)=8>6=\sigma(6,2)$.
	Since $\tau_3\in \A$, the map $\varepsilon$ is well-defined on $\tau_3$. Because $\tau_1$ is the compatible neighbour of $\tau_3$ from above,
	Lemma \ref{L:lemma10} (2) guarantees that the terminating row of $\tau_1$ (row $5$) must be above the starting row of $\tau_3$ (row $3$).	
		
	We remark that $(i)$ and $(ii)$ in Lemma \ref{L:lemma10} are sufficient but not necessary conditions for $\varepsilon$ to be well-defined. For instance,
	$\tau_4$ satisfies neither $(i)$ nor $(ii)$. Nevertheless, the  map $\varepsilon$ is well-defined on $\tau_4$.
\end{example}
We are now in a position to make significant use of the operators $\varepsilon(\tau)$ to establish Theorem \ref{thm3.4}.

{\em Proof of Theorem \ref{thm3.4}}.
Let $\sigma^i$ be $\sigma$ restricted to column $i$. The filling $\sigma^i\cup\sigma^{i+1}$ can be decomposed into a sequence of overlapping descent blocks and neutral blocks vertically, as defined in (\ref{E:odot}). The map $\phi_i$ is given as follows: For the trivial case where there is no descent block; that is, $\ndes(\sigma^{i+1})=\ndes(\sigma^i)$, we set $\phi_i(\sigma)=\sigma$, which clearly satisfies (\ref{E:majthm62})--(\ref{E:ndesthm62}).
	
	
	For the other non-trivial cases, we find all $j$ such that $\CMcal{Q}(\sigma(j,i),\sigma(j-1,i),\sigma(j-1,i+1))=\CMcal{Q}(\sigma(j,i+1),\sigma(j-1,i),\sigma(j-1,i+1))$, and we divide $\sigma^i\cup \sigma^{i+1}$
	into a sequence of components by drawing a line between rows $j$ and $j-1$ for each such $j$. 
	Let $\S_m$ be the $m$th component from top to bottom and let $k$ be the number of components of $\sigma^i\cup \sigma^{i+1}$. Then, $\sigma^i \cup \sigma^{i+1}$ is the disjoint union of the $\S_i$'s for $1\le i\le k$. See Example \ref{Example:5.4}.
	
	For any integer $m$ with $1\le m\le k$, suppose that $\S_m$ is the segment of $\sigma^i \cup \sigma^{i+1}$ from row $\kappa_1$ through row $\kappa_2$ with $\kappa_1\le \kappa_2$. Consider two blocks on the borders of $\S_m$, and let $\tau_2$ (resp. $\tau_1$) be the block of $\sigma^i\cup\sigma^{i+1}$ between rows $\kappa_2$ and $\kappa_2+1$ (resp. between rows $\kappa_1-1$ and $\kappa_1$). The classification of descent blocks in Definition \ref{Def:3} together with Remark \ref{Rem:1} indicates that the boundary blocks are neutral or descent blocks of type $\A$ or $\B$, whereas all blocks contained in $\S_m$ are neutral or descent blocks of type $\mathcal{C}$.

	We denote by $\T_m$ the set of descent blocks from row $\kappa_2+\chi(\tau_2\in\A)$ through row $\kappa_1-\chi(\tau_1\in\B)$, and let $\pi_m$ be the topmost descent block therein. Define $\phi_i(\sigma)$ to be the filling derived from $\sigma$ by applying $t_i^{[\kappa_1,\kappa_2]}$ for all $m$ whenever $|\T_m|$ is odd. See Example \ref{Example:5.4}. 
	
	We shall see  that $\phi_i$ is an involution with properties (\ref{E:majthm62})--(\ref{E:ndesthm62}). 
	First, we verify that 
	 \begin{align}\label{E:pim}
	 \varepsilon(\pi_m)=t_i^{[\kappa_1,\kappa_2]}. 
	 \end{align}
	 If $\pi_m\in\A$, then $\pi_m=\tau_2$ and $\varepsilon(\tau_2)=t_i^{[\kappa_1,\kappa_2]}$ by definition. Otherwise, if $\pi_m\in\B\cup\mathcal{C}$, then the terminating row of $\varepsilon(\pi_m)$ must be row $\kappa_1$. It remains to prove (\ref{E:kappa}) for $x=\kappa_2$. In this case, $\tau_2\not\in\A$, because otherwise $\pi_m=\tau_2\in\A$, which is a contradiction. As a result, $\tau_2$ is neutral or $\tau_2\in \B$ is the neighbour of $\pi_m$ from above. 
	 
	 Without loss of generality, we assume that $\pi_m$ is a left-descent block. If $\tau_2$ is neutral, then (1) of Lemma \ref{L:lemma10} says that $\sigma(\kappa_2,i)>\sigma(\kappa_2,i+1)$, thus confirming that $\CMcal{Q}(\sigma(\kappa_2+1,i),\sigma(\kappa_2,i),\sigma(\kappa_2,i+1))=\CMcal{Q}(\sigma(\kappa_2+1,i+1),\sigma(\kappa_2,i),\sigma(\kappa_2,i+1))=0$. It follows that the starting row of $\varepsilon(\pi_m)$ must be row $\kappa_2$. Otherwise, $\tau_2\in\B$ is the neighbour of $\pi_m$ from above. By (2) of Lemma \ref{L:lemma10}, the starting row of $\varepsilon(\pi_m)$ is below the terminating row $\kappa_2+1$ of $\varepsilon(\tau_2)$. On the other hand, the starting row of $\varepsilon(\pi_m)$ is above the bottom row of $\pi_m$ since $\pi\in \B\cup \mathcal{C}$. This implies that the starting row of $\varepsilon(\pi_m)$ must be row $\kappa_2$.  Consequently, (\ref{E:pim}) holds in both cases. 
	 
	 Notice that $\phi_i$ is a product of commuting $\varepsilon(\pi_m)$, because all components $\S_j$ are disjoint. Thus, each row of $\sigma^i\cup\sigma^{i+1}$ is either fixed or reversed exactly once by $\varepsilon(\pi_m)$ in (\ref{E:pim}). Next, we show that $\phi_i$ is an involution. 
	 
	 Let $\tau=\raisebox{-8pt}{\young(ab,cd)}$ be part of $\sigma^i\cup\sigma^{i+1}$ on the border between two components; that is, $\raisebox{-2pt}{\young(ab)}$ is the bottom row of a component, and $\raisebox{-2pt}{\young(cd)}$ is the top row of its neighbouring component. 
	 By construction, $\CMcal{Q}(a,c,d)=\CMcal{Q}(b,c,d)$, which is trivially preserved if row  $\raisebox{-2pt}{\young(cd)}$ is fixed by $\phi_i$. If row $\raisebox{-2pt}{\young(cd)}$ is reversed by $\phi_i$, then we obtain $\CMcal{Q}(a,d,c)=\CMcal{Q}(b,d,c)$, meaning that the ways of dividing $\sigma^i\cup\sigma^{i+1}$ into components are the same before and after $\phi_i$. Furthermore, the type of the descent blocks is maintained by $\phi_i$. That is, if $\tau$ is a descent block, then $\tau\in\mathcal{X}$ if and only if $\phi_i(\tau)\in\mathcal{X}$ for $\mathcal{X}\in\{\A,\B,\mathcal{C}\}$. This follows directly from Definition \ref{Def:3} and the construction of $\phi_i$. Thus, $\kappa_1$, $\kappa_2$, $\T_m$ and $\pi_m$ are invariant under $\phi_i$. Together with (\ref{E:pim}), we conclude that $\phi_i$ is an involution.
	
     Finally, we establish (\ref{E:majthm62})--(\ref{E:ndesthm62}). Note that (\ref{E:majthm62}) is a direct consequence of (\ref{E:maj1}). For $1\le m\le k$, let $\varepsilon(\pi_m)(\sigma)=\pi$ when $|\T_m|$ is odd.
     We claim that
     \begin{align}\label{E:phi1}
     \quinv(\pi)=\quinv(\sigma)+\ndes(\pi^i)-\ndes(\sigma^i).
     \end{align}
     For even $|\T_m|$, there is an equal number of left- and right-descent blocks of $\T_m$.
     The descent blocks of $\T_m$ {\em alternate} from top to bottom in the sense that every two neighbours are incompatible. In other words, if $\R$ (resp. $\CMcal{L}$) denotes a right-descent (resp. left-descent) block, then
     the sequence of descent blocks in $\T_m$ from top to bottom is $(\ldots, \R,\CMcal{L},\R,\ldots)$.
     This holds because $\varepsilon(\pi_m)$ terminates between two compatible neighbours by Lemma \ref{L:lemma10} $(2)$. We further assert that every left (resp. right) descent block of $\T_m$ becomes a right (resp. left) descent block after applying $\varepsilon(\pi_m)$. 
     
     Let $\tau=\raisebox{-8pt}{\young(ab,cd)}$ be a descent block of $\T_m$. If $\varepsilon(\pi_m)$ swaps the entries of both rows of $\tau$, then the assertion is evidently true. Otherwise, if $\varepsilon(\pi_m)$ only swaps the entries $a$ and $b$, we have $\tau=\tau_1\in\B$. 
     In this case, $\raisebox{-8pt}{\young(ba,cd)}$ is incompatible with $\tau$ by the definition of $\B$. If $\varepsilon(\pi_m)$ swaps only the entries $c$ and $d$, then, $\tau=\tau_2\in\A$ and $\raisebox{-8pt}{\young(ab,dc)}$ is also incompatible with $\tau$ by the definition of $\A$. This proves the assertion.
     
     This results in that the sequence $(\ldots, \R,\CMcal{L},\R,\ldots)$ 
     of descent blocks in $\T_m$ is transformed into  $(\ldots,\CMcal{L},\R,\CMcal{L},\ldots)$ by $\varepsilon(\pi_m)$. 
     Similarly, if $\tau$ is a neutral block with two non-descent (resp. descent) columns, then one can show that $\varepsilon(\pi_m)(\tau)$ is also a neutral block with two non-descent (resp. descent) columns. It follows that only the descent blocks can affect the number of non-descents in each column after $\varepsilon(\pi_m)$ is applied.  
     
     Since $\vert \T_m\vert$ is odd, the number of non-descents in column $i$ increases (resp. decreases) by one if and only if $\pi_m$ is a left-descent (resp. right-descent) block. Equivalently, the number of queue inversion triples increases (resp. decreases) by one, according to (\ref{E:quinv1})--(\ref{E:quinv2}) of Lemmas \ref{L:2} and \ref{L:lemma10} (1). That is,
     \begin{align}
      \label{E:qinvdif1}\quinv(\pi)-\quinv(\sigma)&=\chi(\pi_m\textrm{ is left descent})-\chi(\pi_m\textrm{ is right descent}),\\
     \label{E:qinvdif}&=\ndes(\pi^i)-\ndes(\sigma^i).
     \end{align}
     If $|\T_m|$ is even, then the numbers of left- and right-descent blocks of $\T_m$ are equal. Consequently, the numbers of non-descents in columns $i$ and $i+1$ are interchanged by $\phi_i$; that is, (\ref{E:ndesthm62}). In combination with (\ref{E:qinvdif}), we obtain
     \begin{align*}
     \quinv(\phi_i(\sigma))-\quinv(\sigma)=\ndes(\sigma^{i+1})-\ndes(\sigma^{i})=x_{i+1}-x_i.
     \end{align*}
    This justifies (\ref{E:quinvthm62}) for $\nu=\quinv$. What remains to be proved is (\ref{E:quinvthm62}) for $\nu=\inv$, which by (\ref{E:qinvdif1}) is equivalent to showing that
      \begin{align}\label{E:invdif}
      \inv(\pi)-\inv(\sigma)
      =\chi(\pi_m\textrm{ is left descent})-\chi(\pi_m\textrm{ is right descent})
      \end{align}
      for odd $|\T_m|$. Let $\alpha_m$ be the bottommost descent block of $\T_m$. Since $|\T_m|$ is odd and the descent blocks of $\T_m$ alternate,  blocks $\pi_m$ and $\alpha_m$ must be compatible. Consequently, (\ref{E:invdif}) can be rewritten as 
      \begin{align}\label{E:invdif2}
      \inv(\pi)-\inv(\sigma)
      =\chi(\alpha_m\textrm{ is left descent})-\chi(\alpha_m\textrm{ is right descent}).
      \end{align}
      Here, $\alpha_m\in\B\cup\mathcal{C}$, because if $\alpha_m\in\A$, then $\T_m$ would not contain $\alpha_m$ by definition. If $\alpha_m\in\B$; that is, $\tau_1=\alpha_m$, then (\ref{E:inv11}) and (\ref{E:inv12}) of Lemma \ref{L:2} guarantee (\ref{E:invdif2}). 
      Otherwise, $\alpha_m\in\mathcal{C}$. We claim that
      \begin{align}\label{E:sigma1}
      \sigma(\kappa_1,i)<\sigma(\kappa_1,i+1),
      \end{align}
      if $\alpha_m$ is left descent; otherwise,
      \begin{align}\label{E:sigma2}
      \sigma(\kappa_1,i)>\sigma(\kappa_1,i+1).
      \end{align}
      This can be derived in the same manner as in (1) of Lemma \ref{L:lemma10}, and hence we omit the details. On the other hand, $\tau_1\not\in \B\cup \mathcal{C}$, 
      because otherwise $\alpha_m=\tau_1\in \B$ or row $\kappa_1$ would not be the terminating row of $\varepsilon(\pi_m)$. Equivalently, $\tau_1$ is neutral or $\tau_1\in \A$. Together with  (\ref{E:sigma1})--(\ref{E:sigma2}), we are led to (\ref{E:invdif2}) again by (\ref{E:inv11})--(\ref{E:inv12}) of Lemma \ref{L:2}. This gives (\ref{E:quinvthm62}) for $\nu=\inv$, thus completing the proof.
\qed
\begin{remark}\label{Rem:6}
The type of any descent block $\raisebox{-8pt}{\young(ab,cd)}$ remains unchanged under the bijection $\phi_i$ because it is invariant under the swap of $a$ and $b$, or $c$ and $d$, according to Definition \ref{Def:3}. 
\end{remark}

\begin{example}\label{Example:5.4}
Let $\sigma$ be the filling in Example \ref{E5.1}. Since $\CMcal{Q}(4,6,7)=\CMcal{Q}(6,6,7)$, $\CMcal{Q}(6,3,8)=\CMcal{Q}(7,3,8)$ and $\CMcal{Q}(3,6,5)=\CMcal{Q}(8,6,5)$, 
we draw the bold lines below to divide $\sigma$ into disjoint components 
\begin{align*}
	\S_1=\raisebox{-8pt}{\young(86,46)}\,,\,\,\S_2=\raisebox{-2pt}{\young(67)}\,,\,\,\S_3=\raisebox{-2pt}{\young(38)}\,,\,\,\S_4=\raisebox{-8pt}{\young(65,52)}\,.
\end{align*}
Then, $\T_1=\left\{\raisebox{-8pt}{\young(86,46)}\right\}$, $\T_3=\left\{\raisebox{-8pt}{\young(67,38)},\raisebox{-8pt}{\young(38,65)}\right\}$ and $\T_2=\T_4=\varnothing$. 
Since $\T_1$ is the only set with odd cardinality, we apply $t_1^{[5,6]}$ to $\sigma$; that is, $\phi_1(\sigma)=t_1^{[5,6]}(\sigma)$.
\begin{align*}
	\sigma=\begin{ytableau}[] *(green)8 & *(green)6 \\
		*(green)4 & *(green)6 \\
		\hline
		6 & 7 \\
		\hline
		3 & 8 \\
		\hline
		6 & 5 \\
		5 & 2
	\end{ytableau}
	\rightarrow
	\begin{ytableau}[]*(green)6 & *(green)8 \\
		*(green)6 & *(green)4 \\
		\hline
		6 & 7 \\
		\hline
		3 & 8 \\
		\hline
		6 & 5 \\
		5 & 2
	\end{ytableau}=\phi_1(\sigma)\,.
\end{align*}
It is straightforward to check (\ref{E:majthm62})--(\ref{E:ndesthm62}). That is, 
\begin{align*}
\maj(\phi_1(\sigma))&=\maj(\sigma)=18,\\
{N}\des(\sigma)&=(x_1,x_2)=(2,3)=s_1{N}\des(\phi_1(\sigma)),\\
\nu(\phi_1(\sigma))&=\nu(\sigma)+x_2-x_1=\nu(\sigma)+1
\end{align*}
for $\nu\in\{\inv,\quinv\}$. The two descent blocks of $\T_3$ alternate, as claimed in the proof of Theorem \ref{thm3.4}. That is, $\raisebox{-8pt}{\young(67,38)}$ is a left-descent block, whereas $\raisebox{-8pt}{\young(38,65)}$ is a right-descent block.
\end{example}

	

\section{Proofs of Theorems \ref{T:3} and \ref{pro1}} \label{S:6}
The primary purpose of this section is to complete the proof of Theorem \ref{T:3}.
First, we present Theorem \ref{thm8}, which introduces the second bijection $\theta$ towards the desired bijection $\varphi$ for Theorem \ref{T:3}.
Second, we discuss how Theorems \ref{pro1} and \ref{thm8} together yield Theorem \ref{T:3}. Finally, we confirm Theorem \ref{pro1}.

\begin{theorem}[Second part of Theorem \ref{T:3}]\label{thm8}
    If the Young diagram of $\mu$ is a rectangle, then
 there exists a bijection $\theta:\T(\mu)\rightarrow \T(\mu)$ such that $\sigma\sim \theta(\sigma)$ and  $(\inv,\quinv,\maj)\sigma=(\quinv,\inv,\maj)\theta(\sigma)$.
\end{theorem}
\begin{proof}
	For a partition $\mu=(n^m)$ and a filling $\sigma\in\T(\mu)$, let $x_i$ be the number of non-descents in column $i$ of $\sigma$. Then, ${N}\des(\rev(\sigma))=(x_n,\ldots,x_1)$. Define a bijection $\theta:\T(\mu)\rightarrow \T(\mu)$ as a product of the $\phi_i$'s as follows: 
	In the first step, we apply the bijections $\phi_i$ for $i$ from $n-1$ to $1$ to $\rev(\sigma)$ and let $\tau_1=\phi_1\circ\cdots\circ\phi_{n-1}(\rev(\sigma))$. By Theorem \ref{thm3.4}, the number of queue inversion triples increases by
	\begin{align*}
	\nu(\tau_1)-\nu(\rev(\sigma))=\sum_{i=2}^{n}(x_1-x_i)=(n-1)x_1-\sum_{i=2}^{n}x_i
	\end{align*}
	for $\nu\in\{\inv,\quinv\}$ and ${N}\des(\tau_1)=(x_1,x_n,\ldots,x_2)$. In the next step, we apply the bijections $\phi_i$ for $i$ from $n-1$ to $2$ on $\tau_1$ and let $\tau_2=\phi_2\circ\cdots\circ\phi_{n-1}(\tau_1)$, yielding that the number of queue inversion triples increases by $(n-2)x_2-(x_3+\cdots+x_n)$ and ${N}\des(\tau_2)=(x_1,x_2,x_n,\ldots,x_3)$.
	We continue this process until we generate $\tau_{n-1}$ satisfying ${N}\des(\tau_{n-1})=(x_1,x_2,\ldots,x_n)$. Define $\theta(\sigma)=\tau_{n-1}$. From (\ref{E:majthm62}), we have $\maj(\theta(\sigma))=\maj(\sigma)$. Furthermore,
	\begin{align*}
	\nu(\theta(\sigma))&=\nu(\rev(\sigma))+\sum_{i=1}^{n-1}(n-i)x_i-\sum_{i=2}^{n}(i-1)x_i \\
	&=\nu(\rev(\sigma))+\sum_{i=1}^{n}x_i(n-2i+1).
	\end{align*}
	Comparing with (\ref{eq5}), we conclude that
	$(\inv,\quinv)(\sigma)=(\quinv,\inv)(\theta(\sigma))$, as claimed.
\end{proof}
\begin{remark}\label{Tab:counterexample}
	In the table below, we provide a row-equivalency class of a non-rectangular diagram to disprove (\ref{E:T2}) for arbitrary fillings.
	\begin{table}[H]
		\begin{adjustwidth}{-0cm}{0cm}
			\renewcommand\arraystretch{1.8}
			\begin{center}
				\resizebox{0.7\textwidth}{!}{
					\begin{tabular}{c| c c c c c c}%
						
						$[\sigma]$ & $\young(3,412,333)$ & $\young(3,421,333)$  &  $\young(3,124,333)$  &  $\young(3,142,333)$ & $\young(3,241,333)$& $\young(3,214,333)$\\ \hline
						
						$\maj$   & 2   &  2  &  2  & 2 &2 &2  \\ 
						$\inv$   & 0   &  1  &  2  & 1 &2 &3  \\ 

						$\quinv$  &3   & 2  & 2  & 1 &0&1  \\ 
						
				\end{tabular}}
				\label{T:disp}
			\end{center}
		\end{adjustwidth}
	\end{table}
\end{remark}

We proceed by establishing Theorem \ref{T:3} with the help of Theorems \ref{pro1} and \ref{thm8}.

{\em Proof of Theorem \ref{T:3}}.
For any $\sigma\in\T(\mu)$, let $\tau=\gamma(\sigma)$. Consider the rectangular decomposition of $\tau=\tau_1\sqcup\cdots \sqcup\tau_p$, and define $\pi=\theta(\rev(\tau_1))\sqcup\cdots\sqcup \theta(\rev(\tau_p))$.  We shall see that $\varphi(\sigma)=\pi$ satisfying $\pi\sim \sigma$,  $\quinv(\pi)=\inv(\sigma)$ and $\maj(\pi)=\maj(\sigma)$. Since the number of queue inversion triples or inversion triples located in two columns of unequal heights is invariant under $\rho_i^r$ by Lemma \ref{L:2}, Theorem \ref{thm8} ensures that
\begin{align*}
\quinv(\pi)-\quinv(\tau)&=\sum_{i=1}^{p}(\quinv(\theta(\rev(\tau_i)))-\quinv(\tau_i))\\
&=\sum_{i=1}^{p}(\inv(\rev(\tau_i))-\quinv(\tau_i))=-\kappa(\tau).
\end{align*}
Combining with (\ref{eqthm3.6b}) of Theorem \ref{pro1}, we conclude that $\quinv(\pi)=\inv(\sigma)$. Furthermore, $\pi\sim \sigma$ and $\maj(\pi)=\maj(\sigma)$ are direct consequences of Theorems \ref{pro1} and \ref{thm8}, as desired.

\qed

We now focus on proving Theorem \ref{pro1}.
The following lemma treats special cases where $a$ is the unique smallest integer among the entries $a,b,c$ and $d$ of a square, and it plays an essential role in deriving (\ref{eqthm3.6b}) of Theorem \ref{pro1}.
\begin{lemma}\label{lem3.3}
	For a partition $\mu=(n^m)$, let $\sigma\in\T(\mu)$ and let $a,b,c,d$ be $\sigma(m,i)$, $\sigma(m,i+1)$, $\sigma(m-1,i)$ and $\sigma(m-1,i+1)$, respectively. Set $\sigma(0,j)=\infty$ and $\sigma(m+1,j)=0$ for $1\le j\le n$. If $a<d<b\le c$, then
	\begin{align*}
	\quinv(\sigma)+(n-i-1)=\quinv(t_{i}(\sigma))\quad\mbox{ and }\quad
	\maj(\sigma)-1=\maj(t_{i}(\sigma)).
	\end{align*}
	If $a<c<b\le d$, then
	\begin{align*}
	\quinv(\sigma)-(n-i+1)=\quinv(t_{i}(\sigma))\quad\mbox{ and }\quad
	\maj(\sigma)+1=\maj(t_{i}(\sigma)).
	\end{align*}
\end{lemma}
\begin{proof}
	For $c\geq b>d>a$, it is clear that the major index is reduced by one after switching $a$ and $b$. Furthermore, we have $\CMcal{Q}(0,a,b)=\CMcal{Q}(b,c,d)=1$, whereas $\CMcal{Q}(0,b,a)=\CMcal{Q}(a,c,d)=0$, which implies that the number of queue inversion triples between columns $i$ and $i+1$ is invariant after the transposition of $a$ and $b$.  Second, we consider the number of queue inversion triples induced by columns $i$
	or $i+1$ and column $j$ for all $j>i+1$. There are five possible ways to insert $z$, as outlined in Table \ref{Tab:2}.
	\begin{table}
		\begin{center}
			\begin{tabular}{ccccccc}
				&  & $(\CMcal{Q}(a,c,z),\CMcal{Q}(b,d,z))$ &$(\CMcal{Q}(b,c,z),\CMcal{Q}(a,d,z))$\\
				\hline \\[-1.5ex]
				$c\geq b>d>a>z$  &  & $(1,0)$ & $(1,1)$\\[1ex]
				\hline \\[-1.5ex]
				$c\geq b>d\geq z\geq a$ &  & $(0,0)$ & $(1,0)$ \\[1ex]
				\hline \\[-1.5ex]
				$c\geq b>z>d>a$  &  & $(0,1)$ & $(1,1)$ \\[1ex]
				\hline \\[-1.5ex]
				$c\geq z\geq b>d>a$  & & $(0,0)$ & $(0,1)$ \\[1ex]
				\hline \\[-1.5ex]
				$z>c\geq b>d>a$  & & $(1,0)$ & $(1,1)$ \\[1ex]
			\end{tabular}
			\caption{Change in queue inversion triples when $c\geq b>d>a$. }\label{Tab:2}
		\end{center}
	\end{table}
	Evidently, $\CMcal{Q}(a,c,z)+\CMcal{Q}(b,d,z)+1=\CMcal{Q}(a,d,z)+\CMcal{Q}(b,c,z)$ in all cases. As $z$ is an arbitrary element in row $m-1$ and column $j$ for $i+1<j\le n$, the change from $\raisebox{-2pt}{\young(ab)}$ to $\raisebox{-2pt}{\young(ba)}$ increases the number of queue inversion triples by $n-i-1$, as desired. The proof for the case when $d\ge b>c>a$ is similar, and hence we omit the details.
\end{proof}

We are now ready to prove Theorem \ref{pro1}. In Example \ref{Exam:3}, we illustrate how the bijections $\gamma$ and $\varphi$ are applied to a given filling.

{\em Proof of Theorem \ref{pro1}}.
	We prove the statement by induction on the number of rows. In the base case, if $\mu$
	has only one row, we simply take $\gamma(\sigma)=\rev(\sigma)$ and $\kappa(\rev(\sigma))=0$. For the partition $\mu$ with at least two rows, let $\lambda$ be the partition obtained by removing the top row of $\mu$. Assume that $\gamma: \T(\lambda)\rightarrow \T(\lambda)$ is a bijection such that $\gamma(\tau)$ has the desired properties stated in Theorem \ref{pro1}.
	
	For $\sigma\in\T(\mu)$, let the topmost and the second-topmost rows of $\sigma$ be $b=(b_1,\ldots, b_s)$ and $c=(c_1,\ldots, c_n)$, respectively, where $s\leq n$. Let $\tau$ be obtained from $\sigma$ by removing the top row. Then, by the induction hypothesis, the top row of $\gamma(\tau)$ is $\rev(c)=(c_n,\ldots,c_1)$. Define the map $\gamma: \T(\mu)\rightarrow \T(\mu)$  recursively as follows:
	
	If $s=n$, then let $\gamma(\sigma)$ be the filling obtained by reversing $b$ and placing it on top of $\gamma(\tau)$. 
	
	
	If $s<n$, then let $\alpha$ be the filling obtained by adding the sequence $(0,\ldots,0,b_s,\ldots,b_1)$ with $n-s$ zeros, to the top of $\gamma(\tau)$. We then move all the zeros to the right of $\rev(b)=(b_s,\ldots,b_1)$ by the following procedure, starting with the filling $\alpha$ and its rightmost zero:
	\begin{enumerate}
		\item[(1)] If $b_s>c_s$, apply $\phi_{n-s}$ to $\alpha$. If $0$ and $b_s$ are swapped by $\phi_{n-s}$, stop; otherwise, continue applying $t_{n-s}$ to $\phi_{n-s}(\alpha)$ to switch $0$ and $b_s$.
		
		\item[(2)]If $b_s\leq c_s$, apply $\rho_{n-s}\circ\phi_{n-s}$ to $\alpha$. If $0$ and $b_s$ are swapped by $\phi_{n-s}$, continue by performing $t_{n-s}$ on $\rho_{n-s}\circ\phi_{n-s}(\alpha)$; otherwise, stop.
	\end{enumerate}
	
	Since $0$ and $b_s$ are switched by $t_{n-s}$ or $\rho_{n-s}$, the above process produces a new filling in which the rightmost zero of the top row is moved to column $(n-s+1)$. The procedure for this new filling and its rightmost
	zero is then repeated until all the zeros in the top row have been transported to the end. These zeros are removed and the resulting filling is denoted by $\gamma(\sigma)$. In both cases ($s=n$ and $s<n$),
	 the topmost rows of $\sigma$ and $\gamma(\sigma)$ are the reverses of each other. 
	
	We now prove that (\ref{eqthm3.6b})--(\ref{eqthm3.6d}) are satisfied by $\gamma(\sigma)$. For simplicity, we write $\ell(\lambda)=p$, and
	let $m$ be the number of non-descents between rows $p$ and $p+1$ of $\sigma$. That is,
	\begin{align*}
		m=\ndes(\sigma\vert_p^{p+1})=\sum_{i=1}^{s}\chi(b_i\le c_i).
	\end{align*}
	Let $(x_n,\ldots,x_1)$ be the numbers of non-descents in the first $n$ columns of $\alpha$ from left to right. Then the number of non-descents in the first $n$ columns of $\gamma(\tau)$ is $(x_n-1,\ldots,x_{s+1}-1, x_{s}-\chi(b_s\leq c_s),\ldots,x_{1}-\chi(b_1\leq c_1))$. Consequently, the number of non-descents in the first $n$ columns of $\sigma$ is 
	$(x_1,\ldots,x_s, x_{s+1}-1,\ldots,x_{n}-1)$. By definition and the induction hypothesis, 
	\begin{align}\label{eq1}
	\maj(\alpha)-\maj(\sigma)=\sum_{i=s+1}^{n}(p-1-(x_i-1))=\sum_{i=s+1}^{n}(p-x_i).
	\end{align}
	On the other hand, when we add $(n-s)$ zeros to the right of the top row $b=(b_1,\ldots, b_s)$ of $\sigma$, the number of inversion triples increases by $m(n-s)$ since $\CMcal{Q}(b_i,c_i,0)=1$ is equivalent to $b_i\le c_i$. Therefore, \eqref{eq5} gives
	\begin{align}\label{eq7}
	\quinv(\alpha|_{p}^{p+1})&= \inv(\sigma|_{p}^{p+1})+m(n-s)+\sum_{i=1}^{s}\chi(b_i\leq c_i)(2i-n-1)+\sum_{i=s+1}^{n}(2i-n-1).
	\end{align}	
	Let ${\Omega_s=\ytableausetup
	{mathmode,boxframe=normal,boxsize=1.7em}
	\raisebox{3pt}{\begin{ytableau}[] 0&b_s \\  \scriptstyle c_{s+1} & c_{s} \end{ytableau}}}$ 
be a block of $\alpha$. Then, $\Omega_s\not\in \A$.
We analyze the change of the statistics $\quinv$, ${N}\des$ and $\maj$ by a close inspection of steps (1) and (2).
	
	Case (I): If $b_s>c_s$ and $\phi_{n-s}$ swaps $0$ and $b_s$ of $\alpha$, then (\ref{E:majthm62}), (\ref{E:quinvthm62}) and (\ref{E:ndesthm62}) state that
	\begin{align*}
	\quinv(\phi_{n-s}(\alpha))&=\quinv(\alpha)+x_{s}-x_{s+1},\\
	{N}\des(\phi_{n-s}(\alpha))&=(\ldots,x_s,x_{s+1},\ldots),\\
	\maj(\phi_{n-s}(\alpha))&=\maj(\alpha).
	\end{align*}

    Case (II): If $b_s>c_s$ and $\phi_{n-s}$ fixes the entries $0$ and $b_s$ of $\alpha$, we claim that
    $\phi_{n-s}$ also fixes the entries $c_{s+1}$ and $c_s$ in the second topmost row, and $\CMcal{Q}(0,c_{s+1},c_s)\ne \CMcal{Q}(b_s,c_{s+1},c_s)$. 
    
    Suppose $\Omega_s\in\B$. Then, $0$ and $b_s$ must be switched by $\phi_{n-s}$ because the first two rows of $\alpha$ between columns $n-s$ and $n-s+1$ form exactly one descent block $\Omega_s$ and $\varepsilon(\Omega_s)$ is applied according to the rules of $\phi_{n-s}$. This excludes the possibility that $\Omega_s\in\B$ and we are led to $\Omega_s\in\mathcal{C}$. Since $\Omega_s\in\mathcal{C}$, Remark \ref{Rem:1} indicates that $\CMcal{Q}(0,c_{s+1},c_s)\ne \CMcal{Q}(b_s,c_{s+1},c_s)$. As a result, block $\Omega_s$ belongs to one component, thus implying that $\phi_{n-s}$ must also fix the entries $c_{s+1}$ and $c_s$ if it fixes the entries $0$ and $b_s$. That is, $\phi_{n-s}(\Omega_s)=\Omega_s$.

    Additionally, since $\Omega_s\in\mathcal{C}$, we have $c_s<b_s\le c_{s+1}$. In step (1), we continue by applying $t_{n-s}$ to $\phi_{n-s}(\alpha)$, which, by Lemma \ref{lem3.3}, gives
    \begin{align*}
    \quinv(t_{n-s}\circ\phi_{n-s}(\alpha))&=\quinv(\alpha)+s-1+x_{s}-x_{s+1},\\
    {N}\des(t_{n-s}\circ\phi_{n-s}(\alpha))&=(\ldots,x_s,x_{s+1}+1,\ldots),\\
    \maj(t_{n-s}\circ\phi_{n-s}(\alpha))&=\maj(\alpha)-1.
    \end{align*}

	Case (III): For the case when $b_s\leq c_s$, and $0$ and $b_s$ are not interchanged by $\phi_{n-s}$, let $\beta$ denote the topmost descent block between columns $n-s$ and $n-s+1$ of $\alpha$. We write $\varepsilon(\beta)=\rho_{n-s}^{\kappa}$ and note that all blocks above row $\kappa$ must be neutral.
	
	We distinguish the cases $\kappa=p+1$ and $\kappa<p+1$. Following the notations in the proof of Theorem \ref{thm3.4}, if $\kappa=p+1$, then the set $\T_1$ consists of an equal number of left- and right-descent blocks, since $0$ and $b_s$ are not swapped. Otherwise, $\kappa<p+1$, and we apply $\rho_{n-s}$ to $\phi_{n-s}(\alpha)$, which switches $0$ and $b_s$ without crossing row $\kappa$ by the definition of $\varepsilon$. This yields $\T_1=\varnothing$. It follows that the numbers of non-descents in columns $n-s$ and $n-s+1$ of $\phi_{n-s}(\alpha)$ are invariant under $\rho_{n-s}$ in both cases. By Lemma \ref{L:2} (where $\CMcal{Q}(0,0,b_s)=1$), (\ref{E:majthm62}), (\ref{E:quinvthm62}) and (\ref{E:ndesthm62}), we find
	\begin{align*}
		\quinv(\rho_{n-s}\circ\phi_{n-s}(\alpha))&=\quinv(\alpha)-1+x_{s}-x_{s+1},\\
		{N}\des(\rho_{n-s}\circ\phi_{n-s}(\alpha))&=(\ldots,x_s,x_{s+1},\ldots),\\
		\maj(\rho_{n-s}\circ\phi_{n-s}(\alpha))&=\maj(\alpha).
	\end{align*}

	Case (IV): We continue with the notations in Case (III) and in the proof of Theorem \ref{thm3.4}.
	For the case $b_s\leq c_s$, and $0$ and $b_s$ are swapped by $\phi_{n-s}$, we have $\kappa=p+1$ and $\vert \T_1\vert$ is odd. Observe that $\beta\not\in \A$ because otherwise the starting row of $\varepsilon(\beta)$ would not be the top row. 
	Since all blocks above $\beta$ are neutral and $\varepsilon(\beta)=\rho_{n-s}$ with $\beta\not\in \A$, Lemma \ref{L:lemma10} $(1)$ implies that $\beta$ is a right-descent block and $c_{s+1}<c_s$. By the definition of $\varepsilon$, we then have $\CMcal{Q}(b_s,c_{s+1},c_s)=0$ as a result of $\CMcal{Q}(0,c_{s+1},c_s)=1$ and $\CMcal{Q}(b_s,c_{s+1},c_s)\ne \CMcal{Q}(0,c_{s+1},c_s)$. It follows that $c_{s+1}<b_s\le c_s$ and hence both rows of $\Omega_s$ must be swapped by $\rho_{n-s}$. That is, $\phi_{n-s}(\Omega_s)=\rev(\Omega_s)$.
	Because $\CMcal{Q}(0,b_s,0)=0$, Lemma \ref{L:2} and (\ref{E:majthm62})--(\ref{E:ndesthm62}) show that 
		\begin{align*}
		\quinv(\rho_{n-s}\circ\phi_{n-s}(\alpha))&=\quinv(\alpha)+1+x_{s}-x_{s+1},\\
			{N}\des(\phi_{n-s}(\alpha))&=(\ldots,x_s,x_{s+1},\ldots),\\
		\maj(\rho_{n-s}\circ\phi_{n-s}(\alpha))&=\maj(\alpha).
	\end{align*}
	We now describe how the numbers of non-descents change from $\phi_{n-s}(\alpha)$ to $\rho_{n-s}\circ\phi_{n-s}(\alpha)$. 
	Since $\phi_{n-s}$ includes the factor $\rho_{n-s}$ and $\beta$ is a right-descent block, the sequence of descent blocks in $\T_1$ is $(\R,\CMcal{L},\R,\ldots)$ which contains an odd number of blocks and becomes $(\CMcal{L},\R,\CMcal{L},\ldots)$ after applying $\phi_{n-s}$ to $\alpha$.
	In consequence, the number of non-descents in column $n-s$ increases by one, whereas the number in column $n-s+1$ decreases by one after $\rho_{n-s}$ has been applied to $\phi_{n-s}(\alpha)$. That is, 
		\begin{align*}
		{N}\des(\rho_{n-s}\circ\phi_{n-s}(\alpha))&=(\ldots,x_s+1,x_{s+1}-1,\ldots).
	\end{align*}
	We perform $t_{n-s}$ on $\rho_{n-s}\circ\phi_{n-s}(\alpha)$ according to step $(2)$. Considering $c_{s+1}<b_s\le c_s$, the neutral block $\Omega_s$ is transformed into a left-descent block by $t_{n-s}$. Hence, we conclude that
	\begin{align*}
	\quinv(t_{n-s}\circ \rho_{n-s}\circ\phi_{n-s}(\alpha))&=\quinv(\alpha)+1-(s+1)+x_{s}-x_{s+1},\\
	&=\quinv(\alpha)-1-(s-1)+x_{s}-x_{s+1},\\
	{N}\des(t_{n-s}\circ \rho_{n-s}\circ\phi_{n-s}(\alpha))&=(\ldots,x_s,x_{s+1}-1,\ldots),\\
	\maj(t_{n-s}\circ \rho_{n-s}\circ\phi_{n-s}(\alpha))&=\maj(\alpha)+1
	\end{align*}
	by Lemma \ref{lem3.3}.
	
	Let $\alpha^i$ be the filling obtained after the rightmost zero of $\alpha$ is moved from column $n-s$ to column $n-s+i$ for $1\le i\le s$ and set $\alpha^0=\alpha$. We now calculate the changes in $\quinv$, ${N}\des$ and $\maj$ from $\alpha$ to $\alpha^1$ for Cases (I)--(IV) in a unified manner.
	
	The condition for Case (II) is equivalent to $c_s<b_s\le c_{s+1}$ and $\phi_{n-s}(\Omega_s)=\Omega_s$, and that for Case (IV) is tantamount to $c_{s+1}<b_s\le c_{s}$ and $\phi_{n-s}(\Omega_s)=\rev(\Omega_s)$. Define $y_s=\chi(c_s<b_s\le c_{s+1}\,\,\textrm{and}\,\,\phi_{n-s}(\Omega_s)=\Omega_s)-\chi(c_{s+1}<b_s\le c_{s}\,\,\textrm{and}\,\,\phi_{n-s}(\Omega_s)=\rev(\Omega_s))$. Then, 
	\begin{align*}
	\quinv(\alpha^1)-\quinv(\alpha)&=x_s-x_{s+1}+(s-1)y_s-\chi(b_s\le c_s),\\
	{N}\des(\alpha^1)&=(x_n,\ldots,x_{s+2},x_s,x_{s+1}+y_s,x_{s-1},\ldots,x_1,\ldots),\\
	\maj(\alpha^1)-\maj(\alpha)&=-y_s.
	\end{align*}
	We then update the notations by denoting the second top row of $\alpha^1$ by $\rev(c)=(c_n,\ldots,c_1)$, and we define the block $\Omega_{s-1}$ as before. We repeat this process on $\alpha^1$ to produce $\alpha^{2}$, obtaining
	\begin{align*}
		\quinv(\alpha^2)-\quinv(\alpha^1)&=x_{s-1}-(x_{s+1}+y_s)+(s-2)y_{s-1}-\chi(b_{s-1}\le c_{s-1}),\\
		{N}\des(\alpha^2)&=(x_n,\ldots,x_{s+2},x_s,x_{s-1},x_{s+1}+y_s+y_{s-1},\ldots),\\
		\maj(\alpha^2)-\maj(\alpha^1)&=-y_{s-1}.
	\end{align*}
	This further leads to 
		\begin{align*}
		\quinv(\alpha^2)-\quinv(\alpha)&=\sum_{i=s-1}^{s}(x_i-x_{s+1}-\chi(b_i\le c_i))+(s-2)\sum_{i=s-1}^{s}y_i,\\
		\maj(\alpha^2)-\maj(\alpha)&=-y_{s-1}-y_s.
	\end{align*}
  In general, for $2\le j\le s$, we find that 
	\begin{align}\label{eq2}
	\quinv(\alpha^j)-\quinv(\alpha)
	&=\sum_{i=s-j+1}^{s}(x_i-x_{s+1}-\chi(b_i\le c_i))+(s-j)\sum_{i=s-j+1}^{s}y_i.
	\end{align}
	Taking $j=s$ so that the last sum vanishes, we arrive at
	\begin{align}\label{E:eq2}
	\quinv(\alpha^s)-\quinv(\alpha)=\sum_{i=1}^{s}(x_i-\chi(b_i\le c_i))-sx_{s+1}=\sum_{i=1}^{s}x_i-sx_{s+1}-m.
	\end{align}
	Let $z_1=\sum_{i=1}^{s}y_i$. Then,
	\begin{align*}
	{N}\des(\alpha^s)&=(x_n,\ldots,x_{s+2},x_s,\ldots,x_1,x_{s+1}+z_1,\ldots),\\
	\maj(\alpha^s)-\maj(\alpha)&=-z_1.
	\end{align*}
	The topmost row of $\alpha^s$ is $(0,\ldots,0,b_s,\ldots,b_1,0)$.  Let the second-topmost row of $\alpha^s$ be $\rev(d)=(d_{n-1},\ldots,d_{s+1},d_s,\ldots,d_1,d_0)$. We proceed by moving the second-rightmost zero of $\alpha^s$ to the second-rightmost column via steps (1) and (2). Letting $\alpha^{2s}$ be the resulting filling, we deduce a formula analogous to (\ref{eq2}) and (\ref{E:eq2}). First, we have
	\begin{align*}
	\sum_{i=1}^{s}\chi(b_i\le d_i)=m+z_1,
	\end{align*}
	resulting from the fact that the number of non-descents between the topmost two rows of $\alpha$ is increased (resp. decreased) by one if and only if Case (II) (resp. Case (IV)) is true; that is when $y_i=1$ (resp. $y_i=-1$).
	
	Define ${\Gamma_s=\ytableausetup
		{mathmode,boxframe=normal,boxsize=1.65em}
		\raisebox{3pt}{\begin{ytableau}[] 0&b_s \\  \scriptstyle d_{s+1} & d_{s} \end{ytableau}}}$
	and set $y_s'=\chi(d_s<b_s\le d_{s+1}\,\,\textrm{and}\,\,\phi_{n-s-1}(\Gamma_s)=\Gamma_s)-\chi(d_{s+1}<b_s\le d_s\,\,\textrm{and}\,\,\phi_{n-s-1}(\Gamma_s)=\rev(\Gamma_s))$. Then, we apply steps (1) and (2) to switch $b_s$ and $0$. Subsequently, we update the notations by letting the second-top row be $\rev(d)$, and derive the block $\Gamma_{s-1}$ as well as the integer $y'_{s-1}$. We continue with this process to define $z_2=\sum_{i=1}^{s}y'_i$. As a result,
	\begin{align*}
	\quinv(\alpha^{2s})-\quinv(\alpha^s)&=\sum_{i=1}^{s}(x_i-\chi(b_i\le d_i))-sx_{s+2}+z_2,\\
	&=\sum_{i=1}^{s}x_i-sx_{s+2}-m-z_1+z_2.
	\end{align*}
	In addition, we have
	\begin{align*}
	{N}\des(\alpha^{2s})
	&=(x_n,\ldots,x_{s+3},x_s,\ldots,x_1,x_{s+2}+z_2,x_{s+1}+z_1,\ldots),\\
	\maj(\alpha^{2s})-\maj(\alpha)&=-z_1-z_2.
	\end{align*}
    Generally, for $1\le i\le n-s$, let $\alpha^{is}$ denote the filling obtained after the $i$th rightmost zero of $\alpha^{(i-1)s}$ is moved to the $i$th rightmost column. Let $(u_{s+1},u_s,\ldots,u_1)$ be the sequence of entries directly below $(0,b_s,\ldots,b_1)$ in the top row of $\alpha^{(i-1)s}$.
    
    Define ${\Theta_s=\begin{ytableau}[] 0&b_s \\  \scriptstyle u_{s+1} & u_{s} \end{ytableau}}$ 
    and set $z_{i}=\sum_{j=1}^{s}\chi(u_j<b_j\le u_{j+1}\,\,\textrm{and}\,\,\phi_{n-s-i+1}(\Theta_s)=\Theta_s)-\chi(u_{j+1}<b_j\le u_j\,\,\textrm{and}\,\,\phi_{n-s-i+1}(\Theta_s)=\rev(\Theta_s))$. Then,
    \begin{align}\nonumber
    \sum_{j=1}^{s}\chi(b_j\le u_j)&=m+\sum_{j=1}^{i-1}z_j.\\
	\label{E:quinvf}\quinv(\alpha^{is})-\quinv(\alpha^{(i-1)s})&=\sum_{j=1}^{s}x_j-sx_{s+i}-m-\sum_{j=1}^{i-1}z_j+z_i(i-1).
	\end{align}
	Furthermore, the changes in ${N}\des$ and $\maj$ are given by
	\begin{align}\nonumber
	{N}\des(\alpha^{is})&=(x_n,\ldots,x_{s+i+1},x_s,\ldots,x_1,x_{s+i}+z_i,\ldots,x_{s+1}+z_1,\ldots),\\
	\label{E:maj11}
	\maj(\alpha^{is})-\maj(\alpha)&=-\sum_{j=1}^{i}z_j.
	\end{align}
	Take $i=n-s$. Let $\alpha^{(n-s)s}$ be the filling after all the zeros of $\alpha$ are shifted to the end of the top row. Since $\gamma(\sigma)$ is obtained from $\alpha^{(n-s)s}$ by removing all zeros, we have $\quinv(\gamma(\sigma))=\quinv(\alpha^{(n-s)s})$
	and it follows from (\ref{E:quinvf}) that
	\begin{align}
	\nonumber\quinv(\gamma(\sigma))-\quinv(\alpha)&=\quinv(\alpha^{(n-s)s})-\quinv(\alpha)\\
     \nonumber&=\sum_{i=1}^{n-s}\left(\sum_{j=1}^{s}x_j-sx_{s+i}-m-\sum_{j=1}^{i-1}z_j+z_i(i-1)\right)\\
     \label{E:quinv3}&=(n-s)\sum_{i=1}^{s}x_i-s\sum_{i=s+1}^{n}x_{i}-(n-s)m-\sum_{i=1}^{n-s}z_i(n-s-2i+1).
	\end{align}
	In contrast, removing the zeros of $\alpha^{(n-s)s}$ reduces the major index and number of non-descents in columns $s+1$ through $n$. That is,
	\begin{align}
	\nonumber{N}\des(\gamma(\sigma))&=(x_s,\ldots,x_1,x_{n}+z_{n-s}-1,\ldots,x_{s+1}+z_1-1,\ldots),\\
	\label{E:maj12}\maj(\gamma(\sigma))-\maj(\alpha^{(n-s)s})&=-\sum_{i=1}^{n-s} (p-1-(x_{s+i}-1+z_i)).
	\end{align}
	Specializing $i=n-s$ in \eqref{E:maj11}, and combining with \eqref{E:maj12} and \eqref{eq1}, we find that
	\begin{align*}
	\maj(\gamma(\sigma))=\maj(\alpha)-\sum_{i=1}^{n-s} (p-x_{s+i})=\maj(\sigma).
	\end{align*}
    The change in $\quinv$ remains to be discussed. Recall that (\ref{E:nu1}) yields
	\begin{align*}
	\inv(\sigma)&=\inv(\sigma|_p^{p+1})+\inv(\tau)-\inv(c), \\
	\quinv(\alpha)&=\quinv(\alpha|_p^{p+1})+\quinv(\gamma(\tau))-\inv(c).
	\end{align*}
	Substituting these two expressions into \eqref{eq7} and \eqref{E:quinv3} produces
	\begin{align*}
	\nonumber
	\quinv(\gamma(\sigma))&= \inv(\sigma|_{p}^{p+1})+\sum_{i=1}^{s}\chi(b_i\leq
	c_i)(2i-n-1)+\sum_{i=s+1}^{n}(2i-n-1) \\
	\nonumber&\quad\quad+\inv(\tau)+\kappa(\gamma(\tau))-\inv(c)+(n-s)\sum_{i=1}^{s}x_i-s\sum_{i=s+1}^{n}x_{i} \\
	&\quad\quad-\sum_{i=1}^{n-s}z_i(n-s-2i+1)  \\
	&=\inv(\sigma)+\kappa(\gamma(\sigma)).
	\end{align*}
	The last equality holds because of the following consequence of Lemma \ref{L:1}. Let $\sigma=\sigma_1 \sqcup \cdots \sqcup \sigma_q\in \T(\mu)$, where $q$ is the number of distinct parts of $\mu'$. Then,
	\begin{align*}
	\kappa(\gamma(\tau))&=\sum_{i=1}^{s}(x_i-\chi(b_i\leq c_i))(2i-n-1)+\sum_{i=s+1}^{n}(x_i-1)(2i-n-1)\\
	&\quad\quad+\sum_{i=2}^{q}(\quinv(\gamma(\tau)_i)-\inv(\rev(\gamma(\tau)_i)),\\
	\kappa(\gamma(\sigma))&=\sum_{i=1}^{s}x_i(2i-s-1)+\sum_{i=1}^{n-s}(x_{i+s}+z_i-1)(2i-(n-s)-1) \\
	&\quad\quad+\sum_{i=2}^{q}(\quinv(\gamma(\tau)_i)-\inv(\rev(\gamma(\tau)_i)) \\
	&=\sum_{i=1}^{s}x_i(2i-s-1)+\sum_{i=1}^{n-s}(x_{i+s}+z_i)(2i-n+s-1)\\
	&\quad\quad+\sum_{i=2}^{q}(\quinv(\gamma(\tau)_i)
	-\inv(\rev(\gamma(\tau)_i)).
	\end{align*}
	When $s=n$, it is immediate that $\alpha=\gamma(\sigma)$, $\maj(\alpha)=\maj(\sigma)$ and ${N}\des(\sigma_1)={N}\des(\rev(\alpha_1))$, where $\alpha=\alpha_1 \sqcup \cdots \sqcup \alpha_q$, and $\quinv(\alpha)=\inv(\sigma)+\kappa(\alpha)$ follows from the above argument. This completes the inductive proof.
\qed

\begin{example}\label{Exam:3}
	Let $\mu=(7,7,5,5,5,2)$ and let $\sigma\in \T(\mu)$ be defined as below.
	\begin{align*}
	\sigma=\begin{ytableau}[] 5&4\\ 9 & 3& 6& 1& 3\\ 2 & 5 & 9& 4& 8\\ 3 & 9 & 7&  3& 5  \\ 5 & 8 & 4&  6& 4 & 8&  7  \\ 9 &3 & 6  &5 & 2& 10  & 1
	\end{ytableau}
	\end{align*}

	We construct $\gamma(\sigma)$ and $\varphi(\sigma)$ step by step according to Theorems \ref{pro1} and \ref{thm8}. First, we define $\gamma(\sigma)$ inductively from bottom to top, as stated in the proof of Theorem \ref{pro1}. 
	
	Starting with the filling $\sigma\vert_1^1$ of the bottom row, we have $\gamma(\sigma\vert_1^1)=\rev(\sigma\vert_1^1)$. Since both $\sigma\vert_1^1$ and $\sigma\vert_2^2$ are of equal length (this is the case when $s=n=7$), we obtain $\gamma(\sigma\vert_1^2)=\rev(\sigma\vert_1^2)$. Next, we consider the filling $\sigma\vert_3^3$ that contains fewer boxes than $\sigma\vert_2^2$ (this is the case when $s=5<n=7$). Let $\alpha$ be the filling obtained by placing the sequence $(0,0,5,3,7,9,3)$ on top of $\rev(\sigma\vert_1^2)$.
	Then, we move the two zeros of $\alpha$ to the end, according to steps (1) and (2), as follows, where each time the two columns under consideration are highlighted in green.	
	\begin{align*}
	\alpha=\,\,&\begin{ytableau}[] 0& *(green)0 & *(green)5 & 3 & 7& 9& 3\\ 7&*(green)8 &*(green)4 & 6 & 4& 8& 5\\ 1& *(green)10 &*(green)2 & 5 & 6& 3& 9 \end{ytableau}
	\rightarrow \begin{ytableau}[] 0&5 &*(green)0 &*(green) 3 & 7& 9& 3\\ 7&4 &*(green)8 & *(green)6 & 4& 8& 5\\ 1& 2 &*(green)10 & *(green)5 & 6& 3& 9 \end{ytableau}
	\rightarrow \begin{ytableau}[] 0&5 &3 &*(green) 0 & *(green)7& 9& 3\\ 7&4 &8 & *(green)6 & *(green)4& 8& 5\\ 1& 2 &5 & *(green)10 &*(green) 6& 3& 9 \end{ytableau} \\
	&\rightarrow \begin{ytableau}[] 0&5 &3 & 7 & *(green)0&*(green) 9& 3\\ 7&4 &8 & 6 & *(green)4&*(green) 8& 5\\ 1& 2 &5 & 10 &*(green) 6& *(green)3& 9 \end{ytableau}
	\rightarrow \begin{ytableau}[] 0&5 &3 & 7 & 9&*(green) 0&*(green) 3\\ 7&4 &8 & 6 & 8&*(green) 4&*(green) 5\\ 1& 2 &5 & 10 & 3& *(green)6& *(green)9 \end{ytableau}
	\rightarrow \begin{ytableau}[] *(green)0&*(green)5 &3 & 7 & 9& 3& 0\\ *(green)7&*(green)4 &8 & 6 & 8& 4& 5\\*(green) 1&*(green) 2 &5 & 10 & 3& 6& 9 \end{ytableau}\\
	&\rightarrow \begin{ytableau}[] 5&*(green)0 &*(green)3 & 7 & 9& 3& 0\\ 4&*(green)7 &*(green)8 & 6 & 8& 4& 5\\ 1&*(green) 2 &*(green)5 & 10 & 3& 6& 9 \end{ytableau}
	\rightarrow \begin{ytableau}[] 5&3 &*(green)0 &*(green) 7 & 9& 3& 0\\ 4&7 &*(green)8 &*(green) 6 & 8& 4& 5\\ 1& 2 &*(green)5 & *(green)10 & 3& 6& 9 \end{ytableau}
	\rightarrow \begin{ytableau}[] 5&3 &7 &*(green) 0 & *(green)9& 3& 0\\ 4&7 &6 &*(green) 8 & *(green)8& 4& 5\\ 1& 2 &10 & *(green)5 &*(green) 3& 6& 9 \end{ytableau}\\
	&\rightarrow \begin{ytableau}[] 5&3 &7 & 9 & *(green)0&*(green) 3& 0\\ 4&7 &6 & 8 & *(green)8&*(green) 4& 5\\ 1& 2 &10 & 5 &*(green) 3& *(green)6& 9 \end{ytableau}
	\rightarrow \begin{ytableau}[] 5&3 &7 & 9 & 3& 0& 0\\ 4&7 &6 & 8 & 4& 8& 5\\ 1& 2 &10 & 5 & 6& 3& 9 \end{ytableau}
	\end{align*}
	We remove the last two zeros and add the sequences $(8,4,9,5,2)$, $(3,1,6,3,9)$ and $(0,0,0,4,5)$ successively to the top, yielding the filling below.
	\begin{align*}
	\begin{ytableau}[]0&0 &0 & 4 & 5\\  3& 1& 6& 3 &9\\ 8&4&9&5&2   \\ 5&3 &7 & 9 & 3\\ 4&7 &6 & 8 & 4& 8& 5\\ 1& 2 &10 & 5 & 6& 3& 9 \end{ytableau}
	\end{align*}
	We transport the rightmost zero from column $3$ to column $5$.
	\begin{align*}
	\begin{ytableau}[]*(green)0 &*(green) 4&5 \\*(green)  6&*(green)3&9\\ *(green)9&*(green)5&2   \\ *(green)7 & *(green)9&3\\*(green) 6 & *(green)8&4\\*(green) 10 &*(green) 5&6 \end{ytableau}\rightarrow \begin{ytableau}[]4 &*(green) 0&*(green)5 \\  6&*(green)3&*(green)9\\ 9&*(green)5&*(green)2   \\ 7&*(green)9&*(green)3\\ 6 &*(green) 8&*(green)4\\ 5 &*(green) 10&*(green)6 \end{ytableau}
	\rightarrow \begin{ytableau}[]4 & 5&0 \\  6&9&3\\ 9&2&5   \\ 7&3&9\\ 6 & 8&4\\ 5 & 10&6 \end{ytableau}
	\end{align*}
	Then, we move the zero from column $2$ to column $4$.
	\begin{align*}
	\begin{ytableau}[]*(green)0 & *(green)4&5 \\ *(green) 1&*(green)6&9\\ *(green)4&*(green)9&2   \\ *(green) 3&*(green)7&3\\ *(green)7 & *(green)6&8\\*(green) 2 &*(green)5&10 \end{ytableau}\rightarrow \begin{ytableau}[]4 &*(green)0&*(green)5 \\ 1&*(green)6&*(green)9\\ 4&*(green)9&*(green)2   \\ 3&*(green)7&*(green)3\\ 7&*(green)6&*(green)8\\ 2&*(green)5&*(green)10 \end{ytableau}
	\rightarrow \begin{ytableau}[]4 &5&0 \\ 1&6&9\\ 4&9&2   \\ 3&3&7\\ 7&8&6\\ 2&10&5 \end{ytableau}
	\end{align*}
	Finally, we move the zero from column $1$ to column $3$.
	\begin{align*}
	\begin{ytableau}[]*(green)0 & *(green)4&5 \\*(green)3&*(green)1&6\\ *(green)8&*(green)4&9   \\ *(green) 5&*(green)3&3\\ *(green)4 &*(green) 7&8\\ *(green)1 &*(green)2&10 \end{ytableau}\rightarrow \begin{ytableau}[]4 & *(green)0&*(green)5 \\  3&*(green)1&*(green)6\\ 4&*(green)8 &*(green)9  \\ 3&*(green)5&*(green)3\\ 7&*(green)4&*(green)8\\ 1&*(green)2&*(green)10 \end{ytableau}
	\rightarrow \begin{ytableau}[]4 & 5&0 \\  3&6&1\\ 4&8 &9  \\ 3&3&5\\ 7&8&4\\ 1&10&2 \end{ytableau}
	\end{align*}
	We remove these zeros to generate $\gamma(\sigma)$.
	\begin{align*}
	\gamma(\sigma)=\begin{ytableau}[]4 & 5\\  3& 6& 1& 9 &3\\ 4&8&9&2&5   \\ 3&3 &5 & 7 & 9\\ 7&8 &4 & 6 & 4& 8& 5\\ 1& 10 &2 & 5 & 6& 3& 9 \end{ytableau}
	\end{align*}
	Second, we construct $\varphi(\sigma)$ by applying $\theta$ to each rectangle of $\gamma(\sigma)$ independently, where $\theta$ is defined in the proof of Theorem \ref{thm8}.
	\begin{align*}
	\begin{ytableau}[]4 & 5\\  3& 6\\ 4&8   \\ 3&3 \\ 7&8 \\ 1& 10  \end{ytableau}\rightarrow\begin{ytableau}[]4 & 5\\  6& 3\\ 8&4   \\ 3&3 \\ 7&8 \\ 10& 1 \end{ytableau}
	\quad \begin{ytableau}[]1&*(green)9 &*(green) 3\\ 9&*(green) 2&*(green) 5\\5&*(green) 7&*(green)9   \\4&*(green) 6&*(green)4 \\2&*(green)5&*(green)6 \end{ytableau}
	\rightarrow\begin{ytableau}[]*(green)1&*(green)3 & 9\\*(green) 9&*(green)5& 2\\*(green)5&*(green)7&9   \\*(green)4&*(green) 4&6 \\*(green)2&*(green)6&5 \end{ytableau}
	\rightarrow\begin{ytableau}[]1&*(green)3 &*(green) 9\\ 5&*(green)9& *(green)2\\5&*(green)7&*(green)9   \\4& *(green)4&*(green)6 \\6&*(green)2&*(green)5 \end{ytableau}
	\rightarrow\begin{ytableau}[]1&3 & 9\\ 5&9& 2\\5&7&9   \\4& 4&6 \\6&2&5 \end{ytableau}
	\quad\begin{ytableau}[]8&5 \\ 3&9\end{ytableau}\rightarrow\begin{ytableau}[]8&5 \\ 9&3\end{ytableau}
	\end{align*}
	Assembling these three rectangles gives $\varphi(\sigma)$, as shown below.
	\begin{align*}
	\varphi(\sigma)=\begin{ytableau}[]4 & 5\\  6& 3& 1& 3 &9\\ 8&4&5&9&2   \\ 3&3 &5 & 7 & 9\\ 7&8 &4 & 4 & 6& 8& 5\\ 10& 1 &6 & 2 & 5& 9& 3 \end{ytableau}
	\end{align*}
	We can calculate that $(\inv,\maj)(\sigma)=(\quinv,\maj)(\varphi(\sigma))=(40,33)$, while $\quinv(\sigma)\ne \inv(\varphi(\sigma))$ as $\quinv(\sigma)=32$ and $\inv(\varphi(\sigma))=34$. 
\end{example}

\section*{Acknowledgements}
 Both authors would like to thank the referees for their valuable suggestions and insightful comments. A short version of this manuscript was presented at the conference of FPSAC (Formal Power Series and Algebraic Combinatorics) in July 2025. 

The first author was supported by the Austrian Research Fund FWF Elise-Richter Project V 898-N, and is supported by the Fundamental Research Funds for the Central Universities, Project No. 20720220039. Both authors are supported by the National Nature Science Foundation of China (NSFC), Projects No. 12201529 and No. 12571358.


\end{document}